\documentclass[final,12pt]{elsarticle}

\journal{Discrete Optimization}

\usepackage{amsmath,amssymb,amsfonts,amsthm}
\usepackage{graphicx}
\usepackage{bm}
\usepackage{enumitem,xcolor}
\usepackage[hypertexnames=false,colorlinks=true,breaklinks=true,bookmarks=true,urlcolor=blue,citecolor=blue,linkcolor=blue,bookmarksopen=false,draft=false]{hyperref}
\usepackage{mathrsfs}
\usepackage{tikz}
\usepackage{placeins}
\usepackage{booktabs}
\usepackage{capt-of}
\usepackage{xurl}
\DeclareMathOperator{\conv}{conv}
\DeclareMathOperator{\sign}{sign}

\DeclareMathOperator{\Vol}{Vol}
\DeclareMathOperator{\aff}{aff}
\DeclareMathOperator*{\argmax}{argmax}

\newcommand{\Po}{\mathcal{P}}
\newcommand{\D}{\mathcal{D}}

\theoremstyle{plain}
\newtheorem{theorem}{Theorem}
\newtheorem{lemma}[theorem]{Lemma}
\newtheorem{proposition}[theorem]{Proposition}
\newtheorem{corollary}[theorem]{Corollary}
\theoremstyle{definition}
\newtheorem{remark}[theorem]{Remark}
\newtheorem{example}[theorem]{Example}

\makeatletter
\renewcommand*{\top}{%
  {\mathpalette\@transpose{}}%
}
\newcommand*{\@transpose}[2]{%
  \raisebox{\depth}{$\m@th#1\scriptscriptstyle\mathsf{T}$}%
}
\makeatother

\begin{document}

\begin{frontmatter}

\title{More on Disjunctive Convex Hulls\\ of Generalized Cross Polytopes}

\author[umich]{Yushan Qu}
\author[umich]{Jon Lee\corref{cor1}}
\ead{jonxlee@umich.edu}
\cortext[cor1]{Corresponding author.}
\address[umich]{University of Michigan, Ann Arbor, MI, USA}

\begin{abstract}
We continue the study of the convex hull $\D\subset\mathbb{R}^{d+n}$ of a disjunction of
$n+1$ polytopes in $\mathbb{R}^d$, managed by $n$ binary indicator variables. Previously,
$\D$ had been characterized for arbitrary polytopes when $d\in\{1,2\}$, for polytopes
sharing a common constraint matrix (under a technical condition, including
hyper-rectangles), and for two generalized cross polytopes ($n=1$). Here we give a
complete description of $\D$ for arbitrary $n$ and $d$ when all $n+1$ polytopes are
generalized cross polytopes. Every vertical facet is described by a normalized support
inequality, and such an inequality is facet describing if and only if its equality
graph, a bipartite graph on the polytopes and the coordinates, is connected; all
vertical facets are generated explicitly from feasible spanning trees. We characterize when such a facet is one of the previously
known pairwise inequalities, and we show that pairwise inequalities suffice when
$d\leq3$, but not in general. We establish a bijection between the vertical facets and
the maximal cells of a regular subdivision of $\Delta_n\times\Delta_{d-1}$\,, which yields
the sharp bound $n+1+2^d\binom{n+d-1}{n}$ on the number of facets of $\D$, attained for
generic data. We also give a polynomial-time separation algorithm that returns a
violated facet, which is maximally violated among the normalized support inequalities. Computational experiments show that the facets that are not
pairwise inequalities close a large part of the gap left by the pairwise inequalities.
\end{abstract}

\begin{keyword}
mixed-integer optimization \sep disjunction \sep generalized cross polytope \sep convex hull \sep facet \sep spanning tree \sep regular subdivision \sep separation
\end{keyword}

\end{frontmatter}

\section{Introduction}
We continue the investigation, initiated by \cite{qu2026disjunctionconvexhullsgeneralized},
of the convex hull $\D$ of a disjunction of $n+1$ generalized cross polytopes in $\mathbb{R}^d$ managed by an indicator vector $z\in\mathbb{R}^n$.
\emph{Non-vertical facets} of $\D$ are
those described by linear inequalities that employ only $z$.
Other facets (which employ also $x\in \mathbb{R}^d$ in the describing inequalities)
are known as \emph{vertical facets}.

In \cite{qu2026disjunctionconvexhullsgeneralized}, we gave a
complete description of $\D$ for the case of $n=1$ (i.e., two
generalized cross polytopes). The vertical facets in that case are
referred to as ``pairwise inequalities''. We also demonstrated that
pairwise inequalities, while valid for $n>1$, do not give a complete description of $\D$ in that case. Our main goal in what follows is
to give a complete description for the case of $n>1$.
Additionally, we give a sharp upper bound on the number of facets of $\D$, by connecting our setting to that of regular subdivisions of  $\Delta_{n} \times \Delta_{d-1}$\,.

\medskip
\noindent {\bf Definitions and notation.}
We assume familiarity with some aspects of the theory of polytopes and of integer programming; see \cite{ziegler} and \cite{cornrio,CCZ_Book,NWbook,balas2018disjunctive}, respectively.
We say that $\alpha^\top x \leq \beta$ is
\emph{valid} for a polytope $\Po\subset \mathbb{R}^r$ if it is satisfied by
all points of $\Po$.
For a valid inequality $\alpha^\top x \leq \beta$  of $\Po$, the \emph{face described} is $\Po\cap \{ x \in \mathbb{R}^r ~:~ \alpha^\top x = \beta\}$.
The \emph{dimension} of $\Po$ is one less than the maximum number of affinely-independent points in $\Po$, and in particular, it is
\emph{full dimensional} if its dimension is
the same as the ambient dimension $r$.
A \emph{facet} of $\Po$ is a face having dimension one less than that of $\Po$.
A full-dimensional polytope has an essentially
unique minimal description (up to positive scaling) using
facet-describing inequalities.

Our setup is the same as in \cite{qu2026disjunctionconvexhullsgeneralized}, and we refer the
reader to \cite{QL_DAM2025} for many links to the literature on disjunctive optimization.
For a positive integer $d$, we let $[d]:=\{1,2,\ldots,d\}$.
For a positive integer $n$, we let $N_0 := \{0,1,...,n\}$ and  $N:= \{1,...,n\}$.
We define $n+1$ \emph{generalized cross polytopes}: For
$d \geq 1$ and $\bm{a}^k\in \mathbb{R}^d_{++}$\,, for
$k\in N_0$\,, we consider the generalized cross polytopes
$\Po_k:= \{x \in \mathbb{R}^d: \sum_{j \in [d]} s_{j} \bm{a}^k_{j} x_j \leq 1, \mbox{ for all }
s:=(s_1\,,\cdots,s_d) \in \{\pm1 \}^d\}$.
While each $\Po_k$ is defined via
$2^d$  linear inequalities,
we can also see it as
the solution set of the
single nonlinear inequality
$ \sum_{j \in [d]} \bm{a}^k_{j} |x_j| \leq 1$.

 We employ binary variables $z_k$\,, for
 $k\in N$.
The interpretation of $z_k=1$ is that $x\in\mathbb{R}^d$ is required to be in $\Po_k$\,. Further, if $z_k=0$ for all $k \in N$, then  $x$ is required to be in $\Po_0$\,.
We then require that
$x$ is enforced to be
in (at least) one of the $\Po_k$ using
$
\textstyle \sum_{k\in N} z_k \leq 1.
$
Because the $\Po_k$ may overlap, it can well be that $z_k=0$, but $x\in \Po_k$ for some $k\in N$. Likewise, we can have $z_k=1$ for
 some choice of $k\in N$, but $x\in \Po_0$ or other $\Po_j$\,.

Throughout, whenever we write $\genfrac(){0pt}{1}{x}{z}$,
it can be assumed that
$x\in\mathbb{R}^d$ and $z\in\mathbb{R}^n$.
For $k\in N$,  let $\mathbf{e}_k$ denote the $k$-th standard unit vector in $\mathbb{R}^n$,
and additionally for convenience, let  $\mathbf{e}_0$
denote the zero vector in $\mathbb{R}^n$.
For $k\in N_0$\,, let $\bar{\Po}_k:=\left\{ \genfrac(){0pt}{1}{\hat x}{\mathbf{e}_k}~:~ \hat x\in \Po_k\right\}$,
which is the polytope $\Po_k\subset \mathbb{R}^d$ lifted
to $\mathbb{R}^{d+n}$ by setting $z=\mathbf{e}_k$\,.
For $k\in N_0$\,, let $\mathcal{X}_k$ be the (finite) set of vertices of $\Po_k$\,,
and let $\bar{\mathcal{X}}_k:=\left\{\genfrac(){0pt}{1}{\hat x}{\mathbf{e}_k} ~:~ \hat x\in \mathcal{X}_k\right\}$, the (finite) set of vertices of
$\bar{\Po}_k$\,.
Finally, let
$
\D:= \conv
\left\{
\bar{\Po}_k ~:~ k\in N_0
\right\}
= \conv
\left\{
\bar{\mathcal{X}}_k ~:~ k\in N_0
\right\}.
$
The set of vertices of $\bar{\Po}_k$\,, for $k\in N_0$\,,  is
\[
\bar{\mathcal{X}}_k:=
 \left\{\left( \genfrac{}{}{0pt}{1}{\pm\hat{x}^{k,i}}{\mathbf{e}_k} \right)  ~:~  ~i \in [d]\right\}, \quad\mbox{where } \hat{x}^{k,i}:=\frac{1}{\bm{a}_i^k} \mathbf{e}_i\in\mathbb{R}^d,
\]
which serves to explicitly name the points.

Because each of our generalized cross polytopes is full dimensional,
so is $\D$ (using \cite[Theorem 2]{QL_DAM2025}), and further the (non-vertical) inequalities
$z_j\geq 0$, for $j\in N$, and $\sum_{j\in N} z_j \leq 1$ describe facets of $\D$ (using \cite[Theorems 4 and 5]{QL_DAM2025}).

We note that our generalized cross polytopes are symmetric in each coordinate, about the
origin. In applications, each polytope may have a different center, but for our theory it
suffices to consider the origin for all, via the following simple result.

\begin{proposition}[\cite{qu2026disjunctionconvexhullsgeneralized}]\label{prop:translations}
    Let $t^k \in \mathbb{R}^d$ be arbitrary,
    and let $\Po_k' := \Po_k + t^k$\, for
    $k\in N_0$\,. Let $\D' := {\rm conv}\left( \bar{\Po}'_0\,, \cdots, \bar{\Po}'_n \right)$. We consider the affine bijection $f: \mathbb{R}^{d+n}\rightarrow \mathbb{R}^{d+n}$ defined by $f\genfrac(){0pt}{1}{x}{z} := \genfrac(){0pt}{1}{x + \sum_{j\in N} z_j t^j + \left(1-\sum_{j\in N} z_j\right)t^0}{ z}$. Then, $\D' = f(\D)$.
\end{proposition}

\medskip
\noindent {\bf Organization and contributions.}
In \S\ref{sec:prior}, we summarize the results on pairwise inequalities that we build on.
In \S\ref{sec:spanning}, we show that a normalized support inequality is facet describing
if and only if its equality graph is connected (Theorem \ref{thm:connectedness}), we
generate all vertical facets explicitly from feasible spanning trees
(Corollary \ref{cor:treefacets}), and we obtain a complete description of $\D$ for general
$n$ and $d$ (Theorem \ref{thm:completedescrip}).
In \S\ref{sec:pairwise}, we characterize which facets are pairwise inequalities
(Theorem \ref{thm:conditionforpairwise}); in particular, the pairwise inequalities
describe $\D$ when $n=1$ (Corollary \ref{cor:n1}) or $d\leq3$ (Corollary \ref{cor:d3}),
but not in general (Example \ref{ex:nonpairwise}).
In \S\ref{sec:regular}, we establish a bijection between the vertical facets of $\D$ and
the maximal cells of a regular subdivision of $\Delta_n\times\Delta_{d-1}$
(Theorem \ref{facetcellbijection}), and we use it to prove a sharp upper bound on the
number of facets of $\D$, attained for generic data (Theorem \ref{thm:facetcount}); we
also describe how the facets change as the data cross a wall on which this subdivision is
not a triangulation (Theorem \ref{thm:flip}).
In \S\ref{sec:sep}, we give a polynomial-time separation algorithm that either certifies
that a point belongs to $\D$ or returns a violated facet-describing inequality, which is
maximally violated among the normalized support inequalities
(Theorem \ref{thm:separation}).
In \S\ref{sec:experiments}, we report computational experiments showing that the facets
of $\D$ that are not pairwise inequalities close a substantial part of the relaxation gap.

\section{Pairwise inequalities}\label{sec:prior}
In \cite{qu2026disjunctionconvexhullsgeneralized}, we introduced
``pairwise inequalities'' and used them for a complete description of $\D$ for $n=1$. We briefly summarize that here.

Consider $k_1\,,k_2 \in N_0$ and $k_1 < k_2$\,. For each $j \in [d]$, we define $r^{k_1\,,k_2}_j := \bm{a}^{k_1}_j /\bm{a}^{k_2}_j$\,,
we partition $[d]$, depending on the $r^{k_1\,,k_2}_j$\,, into the blocks:
\begin{align*}
 \mathcal{L}^{k_1\,,k_2}(j) &:= \left\{i \in [d]: r^{k_1\,,k_2}_i < r^{k_1\,,k_2}_j\right\}, \\
 \mathcal{E}^{k_1\,,k_2}(j) &:= \left\{i \in [d]: r^{k_1\,,k_2}_i = r^{k_1\,,k_2}_j\right\}, \\
 \mathcal{G}^{k_1\,,k_2}(j) &:= \left\{i \in [d]: r^{k_1\,,k_2}_i > r^{k_1\,,k_2}_j\right\}.
\end{align*}
 Next, for each $i\in[d]$, we define
\[
\alpha^{k_1\,,k_2}_i(j):=
\begin{cases}
\bm{a}_i^{k_1}=  r^{k_1\,,k_2}_j \bm{a}_i^{k_2}\,, & \mbox{for } i \in \mathcal{E}^{k_1\,,k_2}(j); \\
 \bm{a}_i^{k_1}\,, & \mbox{for } i \in \mathcal{L}^{k_1\,,k_2}(j); \\
 r^{k_1\,,k_2}_j \bm{a}_i^{k_2}\,, & \mbox{for } i \in \mathcal{G}^{k_1\,,k_2}(j).
\end{cases}
\]
Equivalently, $\alpha^{k_1\,,k_2}_i(j):= \min \{ \bm{a}_i^{k_1}\,,  r^{k_1\,,k_2}_j \bm{a}_i^{k_2}\}$. Finally, for each $k \in N_0$\,, we define $M^{k_1\,,k_2}_k(j) := \max_{ i \in [d]}\alpha^{k_1\,,k_2}_i(j)/\bm{a}^k_i $\,.

Now, for
$k_1\,,k_2 \in N_0$\,, with $k_1 < k_2$\,,  $j\in[d]$, and  $s\in \{\pm1 \}^d$,
we define the \emph{pairwise inequality}
\begin{equation}\tag{$F(k_1\,,k_2\,,j,s)$}\label{Fkalpha}
\sum_{i\in [d]} s_i \alpha^{k_1\,,k_2}_i(j) x_i + \sum_{k\in N} \left(M^{k_1\,,k_2}_0(j) - M^{k_1\,,k_2}_k(j) \right) z_k \leq M^{k_1\,,k_2}_0(j).
\end{equation}
We can also view these inequalities in the more compact nonlinear form
\begin{equation}\tag{$\bar{F}(k_1\,,k_2\,,j)$}\label{Fkalpha_bar}
\sum_{i\in [d]}  \alpha^{k_1\,,k_2}_i(j) |x_i| + \sum_{k\in N} \left(M^{k_1\,,k_2}_0(j) - M^{k_1\,,k_2}_k(j) \right) z_k \leq M^{k_1\,,k_2}_0(j),
\end{equation}
for all $j\in[d]$.

In general, the pairwise inequalities are necessary for describing $\D$, and they are sufficient, together with the inequalities that describe non-vertical facets, when $n=1$:

\begin{theorem}[\cite{qu2026disjunctionconvexhullsgeneralized}]\label{thm:facet}
For $k_1\,,k_2 \in N_0$ with $k_1 < k_2$\,, $j\in[d]$, and  $s\in \{\pm1 \}^d$,
the pairwise inequality \ref{Fkalpha} describes a facet of $\mathcal{D}$.
\end{theorem}

\begin{theorem}[\cite{qu2026disjunctionconvexhullsgeneralized}]\label{thm:complete}
For $n=1$, the pairwise inequalities $F(0,1,j,s)$ for all $j\in[d]$ and $s\in \{\pm1 \}^d$,
together with the inequalities $ 0\leq z_1 \leq 1$
describe $\D$.
\end{theorem}

\cite[Example 21]{qu2026disjunctionconvexhullsgeneralized}
demonstrates that
Theorem \ref{thm:complete} does not extend to $n>1$ (see also Example \ref{ex:nonpairwise}),
and our main goal is to give a complete description of $\D$
for general $n$. Of course our complete description
will have to include the inequalities describing non-vertical facets as
well as the pairwise inequalities.

\section{Equality graphs, spanning trees, and a complete description of \texorpdfstring{$\D$}{D}}\label{sec:spanning}

In this section, we give a complete description of $\D$ for arbitrary $n,d\geq1$
(Theorem \ref{thm:completedescrip}). We first introduce the normalized support
inequalities $F(\alpha,s)$, which are valid for $\D$, and identify the vertices of
$\D$ at which they are tight (Theorem \ref{thm:validity}). After a lemma relating
affine independence of tight vertices to forests (Lemma \ref{lem:forest}), we show
that $F(\alpha,s)$ is facet describing if and only if its equality graph $G(\alpha)$
is connected (Theorem \ref{thm:connectedness}). This criterion identifies facets,
but it does not produce their coefficients. We therefore construct, from each
feasible spanning tree $T$, explicit coefficients $\alpha^T$, and show that
$F(\alpha^T,s)$ is facet describing (Proposition \ref{prop:sptoF}). Conversely, every
spanning tree of the equality graph of a facet-describing $F(\alpha,s)$ is feasible
and recovers $F(\alpha,s)$ (Proposition \ref{sprecover}). Corollary
\ref{cor:treefacets} summarizes this correspondence, and
Theorem \ref{thm:completedescrip} shows that these facets, together with the
non-vertical facets, describe $\D$.

\medskip
\noindent {\bf Graph conventions.}
We regard $N_0$ and $[d]$ as disjoint node sets, and we regard each pair
$(k,i)\in N_0\times[d]$ as an arc joining $k\in N_0$ and $i\in[d]$. Thus, for
$H\subset N_0\times[d]$, the pair $(N_0\sqcup[d],H)$ is a bipartite graph. We say that
$T\subset N_0\times[d]$ is a \emph{spanning tree on} $N_0\sqcup[d]$ if the graph
$(N_0\sqcup[d],T)$ is connected and has no cycles.

\medskip
\noindent {\bf Normalized support inequalities.}
For $\alpha\in\mathbb{R}^d_{++}$ and $k\in N_0$\,, let
$M_k(\alpha):= \max_{i \in [d]} \frac{\alpha_i}{\bm{a}^k_i}$\,. Because each $M_k$ is
positively homogeneous, every $\alpha\in\mathbb{R}^d_{++}$ can be scaled so that
$M_0(\alpha)=1$; we always assume this normalization. For $s\in\{\pm1\}^d$ and
$\alpha\in\mathbb{R}^d_{++}$ with $M_0(\alpha)=1$, the \emph{normalized support
inequality} is
\begin{equation}\tag{$F(\alpha,s)$}\label{FmultiG}
        \sum_{i \in [d]} s_i \alpha_i x_i + \sum_{k \in N}(1 - M_k(\alpha))z_k \leq 1.
\end{equation}
For $k\in N_0$\,, the \emph{active set} is
$I_k(\alpha) := \{i \in [d]: \alpha_i = M_k(\alpha)\bm{a}^k_i\}$, which is nonempty. The
\emph{equality graph} $G(\alpha)$ is the bipartite graph having node bipartition
$N_0\sqcup[d]$ and arc set
\begin{align*}
E(\alpha)&:=\{(k,i)\in N_0 \times [d] ~:~ i\in I_k(\alpha)\}\\
&\phantom{:}=\{(k,i)\in N_0 \times [d] ~:~ M_k(\alpha) = \alpha_i/\bm{a}^k_i\}.
\end{align*}
Finally, for fixed $s \in \{\pm 1\}^d$, we write
\[
\hat{\xi}^{k,i}:=\genfrac(){0pt}{1}{s_i\hat{x}^{k,i}}{\mathbf{e}_k}, \quad\mbox{for } (k,i)\in N_0\times[d],
\]
for the vertex of $\bar{\Po}_k$ whose nonzero $x$-coordinate has sign $s_i$\,.

\begin{theorem}\label{thm:validity}
     For $s \in \{\pm 1\}^d$ and $\alpha \in \mathbb{R}^d_{++}$ with $M_0(\alpha)=1$, the normalized support inequality \ref{FmultiG} is valid for $\D$. Moreover, a vertex of $\D$ satisfies \ref{FmultiG} as an equation if and only if it is $\hat{\xi}^{k,i}$ for some $(k,i)\in E(\alpha)$.
\end{theorem}

\begin{proof}
  It suffices to show that the vertices of $\D$ satisfy \ref{FmultiG}.
For $k\in N_0$\,,
the vertices of $\bar{\Po}_k$ are the $2d$ points $\genfrac(){0pt}{1}{\pm\hat{x}^{k,i}}{\mathbf{e}_k}$
 for $ i \in [d]$. Observing that for $k\in N_0$\,, $M_k(\alpha)$ is simply the maximum objective value of
 $\sum_{i \in [d]} s_i \alpha_i x_i$ on the $2d$ vertices of $\Po_k$\,,
 and that $1-  \sum_{k \in N}(1 - M_k(\alpha))z_k = M_k(\alpha)$ when $z:=\mathbf{e}_k$ (using $M_0(\alpha)=1$ for $k=0$),
 validity follows.
 Moreover, at the vertex $\genfrac(){0pt}{1}{t\,\hat{x}^{k,i}}{\mathbf{e}_k}$, with $t\in\{\pm1\}$, the left-hand side minus the right-hand side of \ref{FmultiG} is $t\,s_i\,\alpha_i/\bm{a}^k_i-M_k(\alpha)$, which is zero if and only if $t=s_i$ and $\alpha_i/\bm{a}^k_i=M_k(\alpha)$, that is, if and only if the vertex is $\hat{\xi}^{k,i}$ and $(k,i)\in E(\alpha)$.
\end{proof}

\begin{lemma}\label{lem:forest}
    Fix $\alpha \in \mathbb{R}^d_{++}$ and $s \in \{ \pm 1\}^d$. Let $H \subset E(\alpha)$.
    Then, the vertices $\{\hat{\xi}^{k,i}: (k,i) \in H\}$
    are affinely independent if and only if $H$ is a forest.
\end{lemma}

\begin{proof}
    The vertices $\{\hat{\xi}^{k,i}: (k,i) \in H\}$ are affinely dependent if and
    only if there exist coefficients $\lambda_{k,i}$\,, not all zero, such that
    \begin{align}
    \sum_{(k,i)\in H} \lambda_{k,i}\hat{\xi}^{k,i}&=0, \label{eq:forest-comb}\\
    \sum_{(k,i)\in H} \lambda_{k,i}&=0. \label{eq:forest-aff}
    \end{align}
    For each $k\in N$, taking the $z_k$-coordinate in \eqref{eq:forest-comb} gives
    $\sum_{i:(k,i)\in H}\lambda_{k,i}=0$. Together with \eqref{eq:forest-aff},
    this implies the same for $k=0$, so
    \begin{equation}\label{eq:forest-rows}
    \textstyle\sum_{i:(k,i)\in H}\lambda_{k,i}=0, \quad\text{for every } k\in N_0\,.
    \end{equation}
    For each $i\in[d]$, taking the $x_i$-coordinate in \eqref{eq:forest-comb} and
    using $s_i\neq0$ gives
    \begin{equation}\label{eq:forest-cols}
    \textstyle\sum_{k:(k,i)\in H}\lambda_{k,i}/\bm{a}_i^k=0, \quad\text{for every } i\in[d].
    \end{equation}
    Because $H\subset E(\alpha)$, $\frac{\alpha_i}{\bm{a}_i^k}=M_k(\alpha)$ for every $(k,i)\in H$. Set $\mu_{k,i}:=M_k(\alpha)\lambda_{k,i}$\,. Because $M_k(\alpha)>0$, this change of variables is invertible.
    Multiplying \eqref{eq:forest-rows} by $M_k(\alpha)$, and, because $\alpha_i>0$ for all $i\in[d]$, using $\frac{1}{\bm{a}_i^k}=\frac{M_k(\alpha)}{\alpha_i}$ in \eqref{eq:forest-cols} and multiplying by $\alpha_i$\,, we obtain
    \[
    \sum_{i:(k,i)\in H}\mu_{k,i}=0,\ \forall k\in N_0\,, \qquad \sum_{k:(k,i)\in H}\mu_{k,i}=0,\ \forall i\in[d].
    \]
    We call a weighting of the arcs of $H$ \emph{balanced} if the incident weights sum to zero at every node. Therefore, the vertices are affinely independent if and only if the zero weighting is the only balanced weighting of $H$.

    Suppose first that \(H\) is a forest. If \(H\) has an arc, then it has a leaf, and the equation at that leaf forces the weight of its unique incident arc to be zero. Deleting this arc leaves a balanced weighting of the remaining forest. Repeating this argument shows that all arc weights are zero.

    Conversely, suppose that \(H\) contains a simple cycle. Because \(H\) is bipartite,
    the cycle is even. Assigning alternating weights \(1\) and \(-1\) to its arcs, and zero to all other arcs, gives a nonzero balanced weighting $\mu$, because each node of the cycle meets exactly two of its arcs, with opposite weights. Then $\lambda_{k,i}:=\mu_{k,i}/M_k(\alpha)$ is nonzero and satisfies \eqref{eq:forest-comb} and \eqref{eq:forest-aff}, so the vertices are affinely dependent. The result follows.
\end{proof}

\begin{theorem}\label{thm:connectedness}
    For $s \in \{\pm 1\}^d$ and $\alpha \in \mathbb{R}^d_{++}$ with $M_0(\alpha)=1$, the normalized support inequality \ref{FmultiG} is facet describing if and only if its equality graph $G(\alpha)$ is connected.
\end{theorem}

\begin{proof}
By Theorem \ref{thm:validity}, the vertices of $\D$ satisfying $F(\alpha,s)$
as an equation are precisely $\hat{\xi}^{k,i}$, for $(k,i)\in E(\alpha)$.

Suppose first that $G(\alpha)$ is connected, and let $T$ be a spanning
tree of $G(\alpha)$. Because $G(\alpha)$ has $|N_0|+|[d]|=n+d+1$
nodes, $T$ has $n+d$ arcs. By Lemma \ref{lem:forest}, the
corresponding $n+d$ tight vertices $\hat{\xi}^{k,i}$, $(k,i)\in T$, are
affinely independent, so the face of $\D$ described by $F(\alpha,s)$ has
dimension at least $n+d-1$. Because $\D$ is full dimensional and
$\alpha\neq0$, this face is proper, so its dimension is exactly $n+d-1$,
and $F(\alpha,s)$ is facet describing.

Conversely, suppose that $F(\alpha,s)$ is facet describing. Choose
$n+d$ affinely independent vertices of the associated facet. By
Theorem \ref{thm:validity}, each is $\hat{\xi}^{k,i}$ for a distinct
$(k,i)\in E(\alpha)$; let $H\subset E(\alpha)$ be the set of these $n+d$ arcs.
By Lemma \ref{lem:forest}, $H$ is a forest. Let $V(H)$ denote the set of
nodes incident to the arcs of $H$, and let $c\ge1$ be the number of
connected components of $(V(H),H)$. Because $H$ is a forest,
$|H|=|V(H)|-c$, and because $V(H)\subset N_0\sqcup[d]$, $|V(H)|\le n+d+1$.
Combining these, $n+d=|H|=|V(H)|-c\le n+d+1-c$, so $c=1$ and
$|V(H)|=n+d+1$. Therefore, $H$ is a spanning tree on $N_0\sqcup[d]$, and
because $H\subset E(\alpha)$, the equality graph $G(\alpha)$ is connected.
\end{proof}

We next describe how a spanning tree gives a normalized support inequality.
Let $T$ be a spanning tree on $N_0\sqcup[d]$, rooted at the node $0\in N_0$\,. We
set $M_0^T:=1.$
For each $k\in N$, write the unique path from $0$ to $k$ in $T$ as
\[
    0=k_0-i_1-k_1-i_2-\cdots-i_m-k_m=k,
\]
where $k_0\,,\ldots,k_m\in N_0$ and $i_1\,,\ldots,i_m\in[d]$. We define
\[
    M_k^T
    :=\prod_{q=1}^m
      \frac{\bm{a}_{i_q}^{k_{q-1}}}{\bm{a}_{i_q}^{k_q}}.
\]
For every $q\in[m]$, this definition gives $
    M_{k_{q-1}}^T \bm{a}_{i_q}^{k_{q-1}}
    =M_{k_q}^T \bm{a}_{i_q}^{k_q}\,.$
Then, the path from $0$ to $k$ determines $M_k^T$ from the
normalization $M_0^T=1$.

For each $i\in[d]$, we next define
\[
    \alpha_i^T:=\min_{k\in N_0}M_k^T \bm{a}_i^k\,.
\]
We say that $T$ is a feasible spanning tree if every arc $(k,i)\in T$
attains the minimum defining $\alpha_i^T$\,, that is $\alpha_i^T=M_k^T \bm{a}_i^k\,,\,\text{ for every }(k,i)\in T.$
Equivalently, $T$ is feasible if $M_k^T \bm{a}_i^k\leq M_\ell^T \bm{a}_i^\ell\,, \text{ for every }(k,i)\in T \text{ and every }\ell\in N_0\,.$

\begin{proposition}\label{prop:sptoF}
    Let $T$ be a feasible spanning tree. Then $\alpha^T\in\mathbb{R}^d_{++}$\,,
    $M_k(\alpha^T) = M^T_k$ for $k \in N_0$\,, and $T \subset E(\alpha^T)$.
    Consequently, for every $s \in \{\pm 1\}^d$, $F(\alpha^T,s)$ is a normalized
    support inequality, and it is facet describing.
\end{proposition}

\begin{proof}
Each $M_k^T$ is a product of positive ratios, so $M_k^T>0$ for $k\in N_0$\,,
and hence $\alpha_i^T=\min_{\ell\in N_0} M_\ell^T \bm{a}_i^\ell>0$ for $i\in[d]$.
In particular, $\alpha_i^T\leq M_k^T \bm{a}_i^k$ for all $k\in N_0$ and $i\in[d]$,
so $M_k(\alpha^T)=\max_{i\in[d]}\alpha_i^T/\bm{a}_i^k\leq M_k^T$ for all $k\in N_0$\,.
Now fix $k\in N_0$\,. Because $T$ is spanning, $k$ is incident to some arc
$(k,i)\in T$, and feasibility of $T$ gives $\alpha_i^T=M_k^T \bm{a}_i^k$\,. Hence
$M_k(\alpha^T)\geq \alpha_i^T/\bm{a}_i^k=M_k^T$\,, and so $M_k(\alpha^T)=M_k^T$\,.
Moreover, for every $(k,i)\in T$, $\alpha_i^T=M_k^T \bm{a}_i^k=M_k(\alpha^T)\bm{a}_i^k$\,,
so $(k,i)\in E(\alpha^T)$; that is, $T\subset E(\alpha^T)$.

Because $M_0(\alpha^T)=M_0^T=1$, $F(\alpha^T,s)$ is a normalized support
inequality, valid for $\D$ by Theorem \ref{thm:validity}. Because $T$ is a
spanning tree on $N_0\sqcup[d]$ contained in $E(\alpha^T)$, the equality graph
$G(\alpha^T)$ is connected, so $F(\alpha^T,s)$ is facet describing by
Theorem \ref{thm:connectedness}.
\end{proof}

\begin{example}\label{ex:infeasible}
Feasibility cannot be dropped from Proposition~\ref{prop:sptoF}. Let $n=d=2$ and
$\bm{a}^0=(1,1)$, $\bm{a}^1=(1,2)$, $\bm{a}^2=(1,3)$. It helps to arrange the numbers
$M^T_k \bm{a}^k_i$ in a table with rows $k\in N_0$ and columns $i\in[d]$:
the minimum of column~$i$ is $\alpha^T_i$\,, and $T$ is feasible exactly
when every arc $(k,i)\in T$ sits at a minimum of its column.
Below, boxed entries are the arcs of $T$.

First take the path $0-1-1-2-2$, that is,
$T=\{(0,1),(1,1),(1,2),(2,2)\}$. Following the path from $0$ gives
$M^T_0=1$, $M^T_1=\bm{a}^0_1/\bm{a}^1_1=1$, and $M^T_2=M^T_1 \bm{a}^1_2/\bm{a}^2_2=\tfrac23$:
\[
\begin{array}{c|@{\hspace{1em}}c@{\hspace{1.4em}}c}
              & i=1 & i=2 \\ \hline
  \rule{0pt}{13pt}k=0 & \boxed{1} & 1 \\[6pt]
  k=1         & \boxed{1} & \boxed{2} \\[6pt]
  k=2         & \tfrac23  & \boxed{2} \\[5pt] \hline
  \rule{0pt}{13pt}\alpha^{T} & \tfrac23 & 1
\end{array}
\]

No boxed entry is a column minimum, so $T$ is not feasible. Here $M(\alpha^T)$
differs from $M^T$: $M(\alpha^T)=(1,\tfrac23,\tfrac23)$,
whereas $M^T=(1,1,\tfrac23)$. Replacing $M^T_1=1$ by $M_1(\alpha^T)=\tfrac23$ changes
row $k=1$ of the table to $(\tfrac23,\tfrac43)$, and the arcs of $E(\alpha^T)$ are the
column minima of this corrected table, namely
$E(\alpha^T)=\{(0,2),(1,1),(2,1)\}$. So $G(\alpha^T)$ is disconnected: the nodes $0$ and $2\in[d]$ are cut off from the rest. By Theorem~\ref{thm:connectedness},
$F(\alpha^T,s)$ is not facet describing for any $s\in\{\pm1\}^2$.

Now take the path $0-2-1-1-2$, that is,
$T'=\{(0,2),(1,2),(1,1),(2,1)\}$. Then
$M^{T'}=(1,\tfrac12,\tfrac12)$:
\vspace{-10pt}
\[
\begin{array}{c|@{\hspace{1em}}c@{\hspace{1.4em}}c}
              & i=1 & i=2 \\ \hline
  \rule{0pt}{13pt}k=0 & 1 & \boxed{1} \\[6pt]
  k=1         & \boxed{\tfrac12} & \boxed{1} \\[6pt]
  k=2         & \boxed{\tfrac12} & \tfrac32 \\[5pt] \hline
  \rule{0pt}{13pt}\alpha^{T'} & \tfrac12 & 1
\end{array}
\]
Every boxed entry is a column minimum, so $T'$ is feasible, and by
Proposition~\ref{prop:sptoF},
\[
\tfrac12 s_1x_1+s_2x_2+\tfrac12 z_1+\tfrac12 z_2\le 1
\]
is facet describing for every $s\in\{\pm1\}^2$. Because
$M_k(\alpha^{T'})=M^{T'}_k$\,, the arcs of $E(\alpha^{T'})$ are exactly
the column minima, which are exactly the boxed entries. Thus
$E(\alpha^{T'})=T'$, and $T'$ is the only spanning tree giving this
facet. \hfill$\clubsuit$
\end{example}

\begin{proposition}\label{sprecover}
    Let \ref{FmultiG} be a normalized support inequality whose equality graph
    $G(\alpha)$ is connected, and let $T \subset E(\alpha)$ be a spanning tree
    of $G(\alpha)$. Then $M^T_k = M_k(\alpha)$ for all $k \in N_0$\,, and
    $\alpha^T = \alpha$. Therefore, every spanning tree of $G(\alpha)$ is
    feasible and determines $F(\alpha,s)$.
\end{proposition}

\begin{proof}
By definition, $M_0^T=1$, and by normalization, $M_0(\alpha)=1$. For $k\in N$,
write the unique path from $0$ to $k$ in $T$ as
$0=k_0-i_1-k_1-i_2-\cdots-i_m-k_m=k.$
For every $q\in[m]$, both $(k_{q-1}\,,i_q)$ and $(k_q\,,i_q)$ are arcs of $T$, and
hence belong to $E(\alpha)$. Then
\[
M_{k_{q-1}}(\alpha)\,\bm{a}_{i_q}^{k_{q-1}}=\alpha_{i_q}=M_{k_q}(\alpha)\,\bm{a}_{i_q}^{k_q}\,,
\]
and therefore
\[
M_{k_q}(\alpha)=M_{k_{q-1}}(\alpha)\,\frac{\bm{a}_{i_q}^{k_{q-1}}}{\bm{a}_{i_q}^{k_q}}\,.
\]
Iterating along the path, starting from $M_0(\alpha)=1$, gives
\[
M_k(\alpha)=\prod_{q=1}^{m}\frac{\bm{a}_{i_q}^{k_{q-1}}}{\bm{a}_{i_q}^{k_q}}=M_k^T\,.
\]

It remains to show that $\alpha^T=\alpha$. For every $i\in[d]$ and
$k\in N_0$\,, the definition of $M_k(\alpha)$ gives
$\alpha_i\leq M_k(\alpha)\bm{a}_i^k=M_k^T \bm{a}_i^k\,.$
Taking the minimum over $k\in N_0$ yields $\alpha_i\leq\alpha_i^T\,.$
Because $T$ is spanning, $i$ is incident to some arc $(k_i\,,i)\in T$.
Because $T\subset E(\alpha)$, $\alpha_i=M_{k_i}(\alpha)\bm{a}_i^{k_i}=M_{k_i}^T \bm{a}_i^{k_i}\,.$
Consequently, $\alpha_i^T=\min_{k\in N_0}M_k^T \bm{a}_i^k \leq M_{k_i}^T \bm{a}_i^{k_i}=\alpha_i\,.$
Thus, $\alpha_i^T=\alpha_i$ for every $i\in[d]$, and hence $\alpha^T=\alpha$.

Finally, for every $(k,i)\in T$,
$\alpha_i^T=\alpha_i=M_k(\alpha) \bm{a}_i^k=M_k^T \bm{a}_i^k\,.$
Therefore, $T$ is feasible, and $F(\alpha^T,s)=F(\alpha,s)$.
\end{proof}

\begin{example}\label{ex:spcalc}
    We illustrate Proposition \ref{sprecover}. Let $n=1$, $d=2$, $(\bm{a}_1^0\,,\bm{a}_2^0)=(1,3)$ and $(\bm{a}_1^1\,,\bm{a}_2^1)=(3,9)$, so that $\Po_1=\frac13\Po_0$\,. Also, let $\alpha=(1,3)$. Then $
    M_0(\alpha)=1$ and $M_1(\alpha) =\max\left\{\frac{\alpha_1}{\bm{a}_1^1}, \frac{\alpha_2}{\bm{a}_2^1}\right\}
    =\frac{1}{3}.$ Consequently,
    $\alpha_i=M_k(\alpha)\bm{a}_i^k\,,\text{ for every }(k,i)\in N_0\times[2].$ Then, $E(\alpha)=N_0\times[2]$ and $G(\alpha)=K_{2,2}$\,. Its four spanning trees are
\[
\begin{aligned}
    T_1&=\{(0,1),(1,1),(0,2)\},&
    T_2&=\{(0,2),(1,2),(0,1)\},\\
    T_3&=\{(0,2),(1,2),(1,1)\},&
    T_4&=\{(0,1),(1,1),(1,2)\}.
\end{aligned}
\]

For each $j\in[4]$, the unique path from $0$ to $1$ in $T_j$ passes
through either the node $1 \in [2]$ or the node $2 \in [2]$. Because $
    \frac{\bm{a}_1^0}{\bm{a}_1^1}
    =\frac{\bm{a}_2^0}{\bm{a}_2^1}
    =\frac{1}{3},$ we have $
    M_0^{T_j}=1,
    M_1^{T_j}=\frac{1}{3}.$ It follows that $
    \alpha_1^{T_j}
    =\min\left\{1,\frac{1}{3}\cdot 3\right\}=1,
    \alpha_2^{T_j}
    =\min\left\{3,\frac{1}{3}\cdot 9\right\}=3.$
Then, $\alpha^{T_j}=(1,3)$ for every $j$, and every arc of $T_j$ satisfies $
    \alpha_i^{T_j}=M_k^{T_j}\bm{a}_i^k$\,, which is the feasibility condition, so each $T_j$ is feasible.
Therefore, $T_1\,,T_2\,,T_3$\,, and $T_4$ are distinct feasible spanning
trees that define the same support inequality \ref{FmultiG}.
\hfill $\clubsuit$
\end{example}

\begin{corollary}\label{cor:treefacets}
    The facet-describing normalized support inequalities of $\D$ are exactly
    the inequalities $F(\alpha^T,s)$, for feasible spanning trees $T$ and
    $s\in\{\pm1\}^d$. Moreover, if $F(\alpha,s)$ is facet describing, then every
    spanning tree of $G(\alpha)$ is feasible and represents $F(\alpha,s)$.
\end{corollary}

\begin{proof}
By Proposition \ref{prop:sptoF}, each $F(\alpha^T,s)$ is facet describing.
Conversely, if $F(\alpha,s)$ is facet describing, then $G(\alpha)$ is
connected by Theorem \ref{thm:connectedness}, and the claim follows from
Proposition \ref{sprecover}.
\end{proof}

\begin{theorem}\label{thm:completedescrip}
    Let $\mathcal{T}$ denote the collection of feasible spanning trees. Then, the normalized support inequalities $F(\alpha^T, s)$ for $T \in \mathcal{T}$, $s \in \{\pm 1\}^d$, together with the non-vertical facet inequalities $\sum_{k \in N}z_k \leq 1, z_k \geq 0\,, \forall k \in N$, give a complete description of the convex hull $\D$.
\end{theorem}

\begin{proof}
Let $\mathcal{Q}$ be the polyhedron defined by the inequalities in the statement.
Because all of these inequalities are valid for $\D$, we have $\D\subset\mathcal{Q}$.
Because $\D$ is full dimensional (see \cite[Theorem 2]{QL_DAM2025}), it is the
intersection of the halfspaces of its facet-describing inequalities, each of which is
unique up to positive scaling. So it suffices to show that every facet-describing
inequality of $\D$ appears, up to positive scaling, in the description of $\mathcal{Q}$.

Let $c^\top x+\sum_{k\in N}\gamma_kz_k\leq\delta$ describe a facet of $\D$. The facet is
the convex hull of its vertices, which lie in $\bigcup_{k\in N_0}\bar{\Po}_k$\,. If the
facet contains no point of $\bar{\Po}_k$ for some $k\in N$, then all of its vertices
satisfy $z_k=0$, so it is contained in, and hence equal to, the facet described by
$z_k\geq0$. Similarly, if it contains no point of $\bar{\Po}_0$\,, then it equals the
facet described by $\sum_{k\in N}z_k\leq1$. In either case, the inequality appears in
the description of $\mathcal{Q}$. So we may assume that the facet contains a point of
every $\bar{\Po}_k$\,, $k\in N_0$\,.

For $k\in N_0$\,, let $h_k(c):=\max_{x\in\Po_k}c^\top x=\max_{i\in[d]}|c_i|/\bm{a}_i^k$\,.
Validity and the assumption give $h_0(c)=\delta$ and $h_k(c)+\gamma_k=\delta$ for all
$k\in N$. If $c=0$, then $\delta=0$ and $\gamma=0$, which is impossible. So $c\neq0$, and
then $h_k(c)>0$ for every $k\in N_0$\,; in particular, $\delta=h_0(c)>0$.

We next show that $c_i\neq0$ for every $i\in[d]$. Suppose that $c_i=0$ for some
$i\in[d]$. At each of the two vertices of $\bar{\Po}_k$ with nonzero $x_i$-coordinate,
the left-hand side equals $\delta-h_k(c)<\delta$, so these vertices are not tight. Hence
every vertex of the facet satisfies $x_i=0$, and the facet lies in the intersection of
the hyperplanes $\{x_i=0\}$ and $\{c^\top x+\gamma^\top z=\delta\}$. These are distinct,
because $c\neq0$ and $c_i=0$, so the facet has dimension at most $n+d-2$, a
contradiction.

Define $\bar{c}:=c/\delta$, $\alpha_i:=|\bar{c}_i|$, and $s_i:=\sign(\bar{c}_i)$ for
$i\in[d]$. Then $\alpha\in\mathbb{R}^d_{++}$\,, $\bar{c}_i=s_i\alpha_i$\,, and
$M_k(\alpha)=h_k(c)/\delta$ for all $k\in N_0$\,; in particular, $M_0(\alpha)=1$. Also,
$\gamma_k/\delta=1-M_k(\alpha)$ for every $k\in N$. Therefore, dividing the
facet-describing inequality by $\delta$ gives $F(\alpha,s)$. By
Corollary \ref{cor:treefacets}, there is a feasible spanning tree $T$ with
$F(\alpha^T,s)=F(\alpha,s)$, which appears in the description of $\mathcal{Q}$.
\end{proof}

\section{Pairwise inequalities as support inequalities}\label{sec:pairwise}

We now relate the pairwise inequalities of Section \ref{sec:prior} to the normalized
support inequalities of Section \ref{sec:spanning}. Throughout, we use the ratios
$r^{k_1\,,k_2}_j$\,, the blocks $\mathcal{L}^{k_1\,,k_2}(j)$, $\mathcal{E}^{k_1\,,k_2}(j)$,
$\mathcal{G}^{k_1\,,k_2}(j)$, the coefficients $\alpha^{k_1\,,k_2}(j)$ and the quantities
$M^{k_1\,,k_2}_k(j)$ from Section \ref{sec:prior}, and the active sets $I_k(\alpha)$ and
the equality graph $G(\alpha)$ from Section \ref{sec:spanning}. By definition,
$M^{k_1\,,k_2}_k(j)=M_k(\alpha^{k_1\,,k_2}(j))$ for $k\in N_0$\,. We write
\[
\tilde{\alpha}^{k_1\,,k_2}(j):=\frac{\alpha^{k_1\,,k_2}(j)}{M^{k_1\,,k_2}_0(j)}
\]
for the normalized coefficient vector, so that $M_0(\tilde{\alpha}^{k_1\,,k_2}(j))=1$.
In Section \ref{sec:prior}, pairwise inequalities were defined for $k_1<k_2$\,, but the
same formulas make sense for any distinct $k_1\,,k_2\in N_0$\,. Because
$\alpha^{k_2\,,k_1}(j)=\min\{\bm{a}^{k_2},\,(\bm{a}^{k_2}_j/\bm{a}^{k_1}_j)\,\bm{a}^{k_1}\}
=\frac{1}{r^{k_1\,,k_2}_j}\,\alpha^{k_1\,,k_2}(j)$ (componentwise), interchanging $k_1$ and
$k_2$ does not change $\tilde{\alpha}^{k_1\,,k_2}(j)$, so we may allow any distinct
$k_1\,,k_2$\,. We say that a normalized support inequality $F(\alpha,s)$ is
\emph{pairwise} if it is a positive multiple of $F(k_1\,,k_2\,,j,s)$ for some distinct
$k_1\,,k_2\in N_0$ and $j\in[d]$.

\begin{proposition}\label{prop:specialcase}
    For distinct $k_1\,,k_2 \in N_0$\,, $j \in [d]$, and $s \in \{\pm 1\}^d$, let
    $\tilde{\alpha}:=\tilde{\alpha}^{k_1\,,k_2}(j)$. Then the pairwise inequality
    \ref{Fkalpha} is $M^{k_1\,,k_2}_0(j)$ times the normalized support inequality
    $F(\tilde{\alpha},s)$, and $G(\tilde{\alpha})$ is connected. In particular, both
    describe the same facet of $\D$.
\end{proposition}

\begin{proof}
Write $\alpha:=\alpha^{k_1\,,k_2}(j)$ and $r_j:=r^{k_1\,,k_2}_j$\,. Because $M_k$ is positively
homogeneous, $M_k(\tilde{\alpha})=M_k(\alpha)/M_0(\alpha)
=M^{k_1\,,k_2}_k(j)/M^{k_1\,,k_2}_0(j)$ for $k\in N_0$\,. Dividing \ref{Fkalpha} by
$M^{k_1\,,k_2}_0(j)>0$ therefore gives exactly $F(\tilde{\alpha},s)$.

For every $i \in [d]$, $\alpha_i \leq \bm{a}^{k_1}_i$ and $\alpha_i \leq r_j \bm{a}^{k_2}_i$\,, with
equality in both at $i=j$. Hence $M_{k_1}(\alpha) = 1$ and $M_{k_2}(\alpha) = r_j$\,, and
\[
I_{k_1}(\alpha) = \mathcal{L}^{k_1\,,k_2}(j)\cup\mathcal{E}^{k_1\,,k_2}(j), \qquad
I_{k_2}(\alpha) = \mathcal{E}^{k_1\,,k_2}(j)\cup\mathcal{G}^{k_1\,,k_2}(j).
\]
Scaling does not change which indices attain a maximum, so
$I_k(\tilde{\alpha})=I_k(\alpha)$ for $k\in N_0$ and $G(\tilde{\alpha})=G(\alpha)$. In
$G(\tilde{\alpha})$, the nodes $k_1$ and $k_2$ are joined through the node
$j\in\mathcal{E}^{k_1\,,k_2}(j)$, and every node $i\in[d]$ is adjacent to $k_1$ or to
$k_2$\,. Every node $l\in N_0\setminus\{k_1\,,k_2\}$ is adjacent to the nodes of the
nonempty set $I_l(\tilde{\alpha})\subset[d]$. Therefore, $G(\tilde{\alpha})$ is
connected, and $F(\tilde{\alpha},s)$ is facet describing by
Theorem \ref{thm:connectedness}.
\end{proof}

\begin{remark}
Proposition \ref{prop:specialcase}, together with Theorem \ref{thm:connectedness}, gives
a new proof of Theorem \ref{thm:facet}.
\end{remark}

\begin{remark}\label{rem:pairwisetree}
The proof of Proposition \ref{prop:specialcase} shows that
\begin{align*}
E(\tilde{\alpha})&=\{(k_1\,,i): i\in \mathcal{L}^{k_1\,,k_2}(j)\cup\mathcal{E}^{k_1\,,k_2}(j)\}\\
&\quad\cup\{(k_2\,,i): i\in \mathcal{E}^{k_1\,,k_2}(j)\cup\mathcal{G}^{k_1\,,k_2}(j)\}\\
&\quad\cup\bigcup_{l\in N_0\setminus\{k_1\,,k_2\}}\{(l,i): i\in I_l(\tilde{\alpha})\}.
\end{align*}
A spanning tree of $G(\tilde{\alpha})$ is obtained by taking the path $k_1-j-k_2$\,,
attaching each $i\in\mathcal{L}^{k_1\,,k_2}(j)$ to $k_1$\,, each
$i\in\mathcal{G}^{k_1\,,k_2}(j)$ to $k_2$\,, and each
$i\in\mathcal{E}^{k_1\,,k_2}(j)\setminus\{j\}$ to one of $k_1\,,k_2$\,, and attaching each
$l\in N_0\setminus\{k_1\,,k_2\}$ to one node $i_l\in I_l(\tilde{\alpha})$. By
Proposition \ref{sprecover}, every such tree represents the pairwise inequality; the
tree is not unique when $|\mathcal{E}^{k_1\,,k_2}(j)|>1$ or $|I_l(\tilde{\alpha})|>1$ for
some $l$.
\end{remark}

Proposition \ref{prop:specialcase} shows that every pairwise inequality is a normalized
support inequality with connected equality graph. The next result characterizes which
normalized support inequalities arise in this way.

\begin{theorem}\label{thm:conditionforpairwise}
Let $s\in\{\pm1\}^d$ and $\alpha\in\mathbb{R}^d_{++}$ with $M_0(\alpha)=1$. Then
$F(\alpha,s)$ is pairwise if and only if there exist two distinct nodes
$k_1\,,k_2\in N_0$ such that $I_{k_1}(\alpha) \cap I_{k_2}(\alpha) \neq \emptyset$ and
$I_{k_1}(\alpha) \cup I_{k_2}(\alpha) = [d]$.
\end{theorem}

\begin{proof}
    $(\Rightarrow)$ If $F(\alpha,s)$ is a positive multiple of $F(k_1\,,k_2\,,j,s)$, then
    comparing the coefficients of $x$ and using $M_0(\alpha)=1$ gives
    $\alpha=\tilde{\alpha}^{k_1\,,k_2}(j)$. By the proof of
    Proposition \ref{prop:specialcase},
    $I_{k_1}(\alpha) \cup I_{k_2}(\alpha) = [d]$ and
    $j \in I_{k_1}(\alpha) \cap I_{k_2}(\alpha)$.

    $(\Leftarrow)$ Suppose that there exist two distinct nodes $k_1\,,k_2\in N_0$ such
    that $I_{k_1}(\alpha) \cap I_{k_2}(\alpha) \neq \emptyset$ and
    $I_{k_1}(\alpha) \cup I_{k_2}(\alpha) = [d]$. Choose
    $j \in I_{k_1}(\alpha) \cap I_{k_2}(\alpha)$, and let $r_j:=r^{k_1\,,k_2}_j$\,. Then
    $\alpha_j = M_{k_1}(\alpha) \bm{a}^{k_1}_j = M_{k_2}(\alpha) \bm{a}^{k_2}_j$\,, so
    $M_{k_2}(\alpha)=r_j M_{k_1}(\alpha)$. Now take any $i \in [d]$. By the definition of
    $M_{k_1}(\alpha)$ and $M_{k_2}(\alpha)$,
    $\alpha_i \leq \min \{ M_{k_1}(\alpha) \bm{a}^{k_1}_i\,,\, M_{k_2}(\alpha) \bm{a}^{k_2}_i\}$, and
    equality holds because $i \in I_{k_1}(\alpha) \cup I_{k_2}(\alpha)$. Hence
    \[
    \alpha_i = M_{k_1}(\alpha) \min \{ \bm{a}^{k_1}_i\,,\, r_j \bm{a}^{k_2}_i\}
    = M_{k_1}(\alpha)\,\alpha^{k_1\,,k_2}_i(j), \quad\mbox{for all } i\in[d].
    \]
    So $\alpha$ is a positive multiple of $\alpha^{k_1\,,k_2}(j)$, and because
    $M_0(\alpha)=1$, $\alpha=\tilde{\alpha}^{k_1\,,k_2}(j)$. By
    Proposition \ref{prop:specialcase}, $F(\alpha,s)$ is pairwise.
\end{proof}

\begin{corollary}\label{cor:n1}
For $n=1$, every facet-describing normalized support inequality is pairwise.
Consequently, Theorem \ref{thm:completedescrip} recovers Theorem \ref{thm:complete}.
\end{corollary}

\begin{proof}
Let $F(\alpha,s)$ be facet describing, so that $G(\alpha)$ is connected by
Theorem \ref{thm:connectedness}. Because $N_0=\{0,1\}$, every node $i\in[d]$ is
adjacent to $0$ or $1$, so $I_0(\alpha)\cup I_1(\alpha)=[d]$. The nodes $0$ and $1$ are
joined by a path, which alternates between $N_0$ and $[d]$, so it has the form
$0-i-1$ for some $i\in I_0(\alpha)\cap I_1(\alpha)$. By
Theorem \ref{thm:conditionforpairwise}, $F(\alpha,s)$ is pairwise, and by the symmetry
noted above, it is a positive multiple of $F(0,1,j,s)$ for some $j\in[d]$. Therefore,
the description in Theorem \ref{thm:completedescrip} consists of the inequalities
$F(0,1,j,s)$ and $0\leq z_1\leq1$, which is Theorem \ref{thm:complete}.
\end{proof}

\begin{corollary}\label{cor:d3}
For $d\leq3$ and every $n\geq1$, every facet-describing normalized support inequality is
pairwise. Consequently, for $d\leq3$, the pairwise inequalities $F(k_1\,,k_2\,,j,s)$, for
$k_1<k_2$ in $N_0$\,, $j\in[d]$ and $s\in\{\pm1\}^d$, together with the non-vertical
inequalities $z_k\geq0$, for $k\in N$, and $\sum_{k\in N}z_k\leq1$, describe $\D$.
\end{corollary}

\begin{proof}
Let $F(\alpha,s)$ be facet describing, so that $G(\alpha)$ is connected by
Theorem \ref{thm:connectedness}. By Theorem \ref{thm:conditionforpairwise}, it suffices
to find distinct $k,k'\in N_0$ with $I_k(\alpha)\cap I_{k'}(\alpha)\neq\emptyset$ and
$I_k(\alpha)\cup I_{k'}(\alpha)=[d]$. If $d=1$, then $I_k(\alpha)=\{1\}$ for every
$k\in N_0$\,, and any two distinct $k,k'$ will do. Let $d\in\{2,3\}$. If
$|I_k(\alpha)|=1$ for every $k\in N_0$\,, then every node of $N_0$ has degree $1$ in
$G(\alpha)$, so each connected component contains at most one node of $[d]$,
contradicting connectedness. Choose $k$ with $|I_k(\alpha)|\geq2$. If
$I_k(\alpha)=[d]$, then any $k'\neq k$ will do, because $I_{k'}(\alpha)\neq\emptyset$.
Otherwise, $d=3$ and $I_k(\alpha)=\{i_1\,,i_2\}$, say. Consider a path in $G(\alpha)$ from
$i_3$ to $i_1$\,, and let $i_3-k'-i'$ be its first two arcs. Then $i'\neq i_3$\,, so
$i'\in I_k(\alpha)$, and $k'\neq k$, because $i_3\notin I_k(\alpha)$. Therefore
$i'\in I_k(\alpha)\cap I_{k'}(\alpha)$ and $I_k(\alpha)\cup I_{k'}(\alpha)=[3]$. The last
statement follows from Theorem \ref{thm:completedescrip} and the symmetry in $k_1\,,k_2$
noted at the beginning of this section.
\end{proof}

Example \ref{ex:nonpairwise} below shows that Corollary \ref{cor:d3} does not extend to
$d=4$.

For $n>1$, the conditions of Theorem \ref{thm:conditionforpairwise} can fail for a
facet, as the next example, which is \cite[Example 21]{qu2026disjunctionconvexhullsgeneralized},
shows.

\begin{example}\label{ex:nonpairwise}
    Let $n=2$, $d=4$, and
    \[\Po_0 = \{ x \in \mathbb{R}^4 ~: \pm x_1 \pm x_2 \pm 10 x_3 \pm 10 x_4 \leq 1\},\]
    \[\Po_1 = \{ x \in \mathbb{R}^4 ~: \pm 10 x_1 \pm \frac{1}{2} x_2 \pm \frac{1}{2} x_3 \pm 10 x_4 \leq 1\},\]
    \[\Po_2 = \{ x \in \mathbb{R}^4 ~: \pm 10 x_1 \pm 10 x_2 \pm \frac{1}{3} x_3 \pm \frac{1}{3} x_4 \leq 1\}.\]
    Choose $\alpha = (1,1,1,1)$ and $s=(1,1,1,1)$. Then
    \[
\renewcommand{\arraystretch}{1.2}
\begin{array}{c@{\qquad}c@{\qquad}c@{\qquad}c}
\toprule
k
  & \left(\alpha_i/\bm{a}_i^k\right)_{i=1}^4
  & M_k(\alpha)
  & I_k(\alpha)\\
\midrule
0 & (1,1,0.1,0.1) & 1 & \{1,2\}\\
1 & (0.1,2,2,0.1) & 2 & \{2,3\}\\
2 & (0.1,0.1,3,3) & 3 & \{3,4\}\\
\bottomrule
\end{array}
\]
    In particular, $M_0(\alpha)=1$, and $F(\alpha,s)$ is
    \[x_1+x_2+x_3+x_4-z_1-2z_2 \leq 1.\]
    Writing $\hat{i}$ for the node $i\in[4]$, the equality graph $G(\alpha)$ is the path
    $\hat{1} - 0 - \hat{2} - 1 - \hat{3} - 2 - \hat{4}$ shown in
    Figure~\ref{fig:nonpairwiseequalitygraphexample}. It is connected, so $F(\alpha,s)$ is
    facet describing by Theorem \ref{thm:connectedness}. However, no pair of active sets
    both intersects and covers $[4]$:
\begin{align*}
  I_0(\alpha)\cap I_1(\alpha)&=\{2\},
  & I_0(\alpha)\cup I_1(\alpha)&=\{1,2,3\}\neq[4],\\
  I_1(\alpha)\cap I_2(\alpha)&=\{3\},
  & I_1(\alpha)\cup I_2(\alpha)&=\{2,3,4\}\neq[4],\\
  I_0(\alpha)\cap I_2(\alpha)&=\varnothing,
  & I_0(\alpha)\cup I_2(\alpha)&=[4].
\end{align*}
    So, by Theorem \ref{thm:conditionforpairwise}, $F(\alpha,s)$ is not pairwise.
\hfill $\clubsuit$
\end{example}

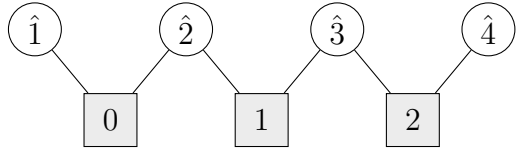
\begin{figure}[htbp]
    \centering
    \begin{tikzpicture}[
        d/.style={
            circle,
            draw,
            minimum size=7mm,
            inner sep=0pt
        },
        N0/.style={
            rectangle,
            draw,
            fill=gray!15,
            minimum size=7mm,
            inner sep=0pt
        }
    ]
        \node[d] (i1) at (0,1.2)   {$\hat{1}$};
        \node[d] (i2) at (2,1.2)   {$\hat{2}$};
        \node[d] (i3) at (4,1.2)   {$\hat{3}$};
        \node[d] (i4) at (6,1.2)   {$\hat{4}$};

        \node[N0] (k0) at (1,0) {$0$};
        \node[N0] (k1) at (3,0) {$1$};
        \node[N0] (k2) at (5,0) {$2$};

        \draw
            (i1) -- (k0) -- (i2) -- (k1)
            -- (i3) -- (k2) -- (i4);
    \end{tikzpicture}
    \caption{The equality graph $G(\alpha)$ of Example~\ref{ex:nonpairwise}. White circles are the nodes $\hat{i}$, $i \in [4]$, and grey squares are the nodes $k \in N_0 = \{0,1,2\}$.}
    \label{fig:nonpairwiseequalitygraphexample}
\end{figure}

\section{Regular subdivisions and the number of facets}\label{sec:regular}

We next relate the normalized vertical facets of $\D$ to the maximal cells of a
regular subdivision of a product of simplices. We use this correspondence to give
a sharp upper bound on the number of facets of $\D$, attained for generic data.

\medskip
\noindent {\bf Logarithmic coordinates.}
Throughout this section, we set $w_{k,i}:=\log \bm{a}^k_i$\,, for $(k,i)\in N_0\times[d]$.
Let $\alpha\in\mathbb{R}^d_{++}$ with $M_0(\alpha)=1$, and set $\eta_i:=\log\alpha_i$ and
$u_k:=\log M_k(\alpha)$, so that $u_0=0$. Because $\alpha_i\leq M_k(\alpha)\bm{a}^k_i$ for all
$(k,i)$, with equality exactly for $(k,i)\in E(\alpha)$, we have
\begin{equation}\tag{$*$}\label{logrelation}
\eta_i - u_k \leq w_{k,i}\,,\quad \mbox{for all } (k,i) \in N_0 \times [d],
\end{equation}
with equality if and only if $(k,i)\in E(\alpha)$.
We will interpret \eqref{logrelation} in terms of the lower faces of a lifted product
of simplices.

\medskip
\noindent {\bf Regular subdivisions of a product of simplices.}
Define
\[
\Delta_n:=
\left\{
p \in \mathbb{R}^{N_0}_+:
\sum_{k\in N_0}p_k=1
\right\},
\qquad
\Delta_{d-1}:=
\left\{
q \in \mathbb{R}^d_+:
\sum_{i\in[d]}q_i=1
\right\}.
\]
For $k\in N_0$\,, let $p^k\in\mathbb{R}^{N_0}$ be the unit vector indexed by $k$, and for
$i\in[d]$, let $q^i\in\mathbb{R}^d$ be the $i$-th unit vector; these are the vertices of
$\Delta_n$ and $\Delta_{d-1}$\,, respectively. The product
$Q:=\Delta_n\times\Delta_{d-1}$ has dimension $n+d-1$ and vertex set
\[
V(Q)=
\left\{
v_{k,i}:=(p^k,q^i):
(k,i)\in N_0\times[d]
\right\}.
\]
We recall the lifting construction for regular subdivisions; see
\cite{DeloeraRambuSantos2010} for more details. A \emph{subdivision} of $Q$ is a finite
collection of polytopes, called \emph{cells}, that covers $Q$, is closed under taking
faces, and has the property that the intersection of any two cells is either empty or
a common face of both. A cell is \emph{maximal} if it is not properly contained in
another cell; the maximal cells are full dimensional in $Q$. A subdivision is a
\emph{triangulation} if every maximal cell is a simplex.

For heights $w\in\mathbb{R}^{N_0\times[d]}$, lift each vertex $v_{k,i}$ to
$(v_{k,i}\,,w_{k,i})$, and let
$\overline{Q}_w:=\conv\left\{(v_{k,i}\,,w_{k,i}):(k,i)\in N_0\times[d]\right\}$.
For an affine function $\ell:\aff(Q) \to \mathbb{R}$ satisfying
$\ell(v_{k,i})\le w_{k,i}$ for all $(k,i)\in N_0\times[d]$, the set
$\conv\{(v_{k,i}\,,w_{k,i}):\ell(v_{k,i})=w_{k,i}\}$ is a \emph{lower face} of
$\overline{Q}_w$\,, and its projection
$C_\ell:=\conv\{v_{k,i}:\ell(v_{k,i})=w_{k,i}\}$ is a cell. These cells, together with
their faces, form the \emph{regular subdivision} $\mathscr{S}(w)$ of $Q$ induced by $w$.
Its maximal cells are the projections of the lower facets of $\overline{Q}_w$\,. We write
$\mathscr{S}:=\mathscr{S}(w)$ when $w_{k,i}=\log \bm{a}^k_i$\,.

\begin{lemma}\label{lem:uniquerepre}
Every affine function $\ell:\aff(Q)\to \mathbb{R}$ can be represented as
$\ell(v_{k,i})=\eta_i-u_k$\,,
for some $u\in \mathbb{R}^{N_0}$ and $\eta\in \mathbb{R}^d$. With $u_0=0$, this representation is unique.
\end{lemma}

\begin{proof}
We may write any affine function on $\aff(Q)$ as $\ell(p,q) = r^\top p + t^\top q + c$.
Because $\sum_{k\in N_0}p_k = 1$ on $\aff(Q)$, we may absorb $c$ into $r$, and so assume
that $\ell(p,q)=r^\top p+t^\top q$. Setting $u_k:=r_0-r_k$ and $\eta_i:=r_0+t_i$ gives
$u_0=0$ and $\ell(v_{k,i})=r_k+t_i=\eta_i-u_k$\,. For uniqueness, the values at $v_{0,i}$\,,
for $i \in [d]$, give $\eta_i = \ell(v_{0,i})$, and then, for each $k \in N_0$\,, the value
at any $v_{k,i}$ gives $u_k = \eta_i - \ell(v_{k,i})$.
\end{proof}

For a nonempty arc set $E \subset N_0\times[d]$, define $C_E:=\conv\{v_{k,i}:(k,i)\in E\}$.
Let $K_E:=\{k\in N_0:(k,i)\in E\text{ for some }i\in[d]\}$ and
$I_E:=\{i\in[d]:(k,i)\in E\text{ for some }k\in N_0\}$, let $G(E)$ be the bipartite graph
with node set $V(E):=K_E \sqcup I_E$ and arc set $E$, and let $c(G(E))$ denote its number
of connected components.

\begin{lemma}\label{lem:dimcell}
For every nonempty $E \subset N_0 \times [d]$, $\dim C_E=|K_E|+|I_E|-c(G(E))-1.$
In particular, $C_E$ is full dimensional in $Q$ if and only if $G(E)$ is connected
and has node set $N_0\sqcup[d]$.
\end{lemma}

\begin{proof}
Let $A_E$ be the matrix whose columns are the vectors $v_{k,i}=(p^k,q^i)$, for $(k,i) \in E$,
restricted to the rows indexed by $K_E\sqcup I_E$ (the other rows are zero); $A_E$ is
the unoriented incidence matrix of $G(E)$. Orienting every arc from $k$ to $i$, let
$B_E$ be the oriented incidence matrix, whose column for $(k,i)$ is
$\genfrac(){0pt}{1}{p^k}{-q^i}$ restricted to the same rows. With
$R := \operatorname{diag}(I_{|K_E|}\,, -I_{|I_E|})$, we have $R A_E = B_E$\,, so
$\operatorname{rank}(A_E) = \operatorname{rank}(B_E^\top)$. For
$y \in \mathbb{R}^{K_E\sqcup I_E}$, the entry of $B_E^\top y$ for the arc $(k,i)$ is
$y_k - y_i$\,, so $B_E^\top y = 0$ if and only if $y$ is constant on each connected
component of $G(E)$. Therefore $\dim\ker(B_E^\top) = c(G(E))$, and
$\operatorname{rank}(A_E) = |K_E| + |I_E| - c(G(E))$. Because every $v_{k,i}$ lies in the
affine hyperplane $\sum_{k \in N_0} p_k = 1$, which does not contain the origin,
$\dim C_E=\operatorname{rank}(A_E)-1$. Finally, $\dim Q=n+d-1$, and
$|K_E|+|I_E|-c(G(E))-1=n+d-1$ holds if and only if $|K_E|+|I_E|=n+d+1$ and
$c(G(E))=1$.
\end{proof}

For a cell $C$ of $\mathscr{S}$, define $E(C):=\{(k,i)\in N_0\times[d]:v_{k,i}\in C\}$ and
$G(C):=G(E(C))$.

\begin{theorem}\label{facetcellbijection}
For each $s\in\{\pm1\}^d$, the map $F(\alpha,s) \longmapsto C(\alpha):= \conv\{v_{k,i}:(k,i)\in E(\alpha)\}$ is a bijection between the facet-describing normalized support inequalities of $\D$ with sign vector $s$ and the maximal cells of $\mathscr{S}$, and $G(C(\alpha))=G(\alpha)$. Moreover, the inverse is explicit: if $C$ is a maximal cell with lower supporting affine
function $\ell_C(v_{k,i})=\eta_i-u_k$\,, where $u_0=0$, then the corresponding facet is $F(\alpha,s)$ with $\alpha_i=e^{\eta_i}$, and $M_k(\alpha)=e^{u_k}$\,.
\end{theorem}

\begin{proof}
We construct two maps and show that they are mutually inverse.

\smallskip
\noindent \emph{From a facet to a cell.} Let $F(\alpha,s)$ be facet describing, and define $\eta$ and $u$ as in \eqref{logrelation}. Because $(p,q)\mapsto \eta^\top q-u^\top p$ is linear, $\ell_\alpha(v_{k,i}) := \eta_i - u_k$ extends to an affine function on $\aff(Q)$, and by \eqref{logrelation}, it lower-supports $\overline{Q}_w$\,, with equality precisely for $(k,i) \in E(\alpha)$. So $C(\alpha)$ is a cell of $\mathscr{S}$. By Theorem \ref{thm:connectedness}, $G(\alpha)$ is connected, and it has node set $N_0 \sqcup [d]$, so, by Lemma \ref{lem:dimcell}, $C(\alpha)$ is full dimensional, hence a maximal cell. This defines a map $\Phi : F(\alpha,s) \mapsto C(\alpha)$, and $G(C(\alpha))=G(E(\alpha))=G(\alpha)$.

\smallskip
\noindent \emph{From a cell to a facet.} Let $C$ be a maximal cell of $\mathscr{S}$. There is an affine function $\ell_C$ with $\ell_C(v_{k,i}) \le w_{k,i}$ for every $(k,i)$, and $v_{k,i} \in C$ if and only if $\ell_C(v_{k,i}) = w_{k,i}$\,. By Lemma \ref{lem:uniquerepre}, we may write $\ell_C(v_{k,i}) = \eta_i - u_k$\,, with $u_0 = 0$. Define $M_k := e^{u_k} > 0$ and $\alpha_i := e^{\eta_i} > 0$; exponentiating gives $\alpha_i \le M_k \bm{a}^k_i$ for every $(k,i)$, with equality if and only if $(k,i)\in E(C)$. Because $C$ is full dimensional, Lemma \ref{lem:dimcell} shows that $G(C)$ is connected with node set $N_0 \sqcup [d]$. In particular, every $k \in N_0$ is incident to some arc of $G(C)$, so $M_k = \max_{i \in [d]} \alpha_i/\bm{a}^k_i = M_k(\alpha)$, and $M_0(\alpha)=e^{u_0}=1$. Therefore $E(\alpha)=E(C)$, and $G(\alpha)=G(C)$ is connected, so $F(\alpha,s)$ is facet describing by Theorem \ref{thm:connectedness}. This defines a map $\Psi : C \mapsto F(\alpha,s)$.

\smallskip
\noindent \emph{$\Phi$ and $\Psi$ are inverse.} We have shown that $\Phi(\Psi(C))=C(\alpha)=\conv\{v_{k,i}:(k,i)\in E(C)\}=C$. Conversely, the lower supporting function of the maximal cell $C(\alpha)$ is uniquely determined, because $C(\alpha)$ is full dimensional, and by Lemma \ref{lem:uniquerepre} its normalized representation is $\eta_i-u_k$ with $\eta=\log\alpha$ and $u=\log M(\alpha)$. Hence $\Psi(\Phi(F(\alpha,s))) = F(\alpha,s)$.
\end{proof}

\begin{remark}
In the logarithmic coordinates of \eqref{logrelation}, the definitions of $M_k(\alpha)$ and
of $\alpha^T$ read $u_k=\max_i(\eta_i-w_{k,i})$ and $\eta_i=\min_k(u_k+w_{k,i})$; these are
min-plus (tropical) matrix--vector products with the matrix $w$. Regular subdivisions of
$\Delta_n\times\Delta_{d-1}$ are dual to arrangements of tropical hyperplanes and
describe the combinatorics of the tropical polytope determined by $w$; see
\cite{DevelinSturmfels2004}. Theorem \ref{facetcellbijection} thus identifies the
vertical facets of $\D$ with the maximal cells of this tropical picture.
\end{remark}

\medskip
\noindent {\bf Cycles.}
A \emph{simple cycle} is a cycle $C: k_1 - i_1 - k_2 - i_2 - \cdots - k_r - i_r - k_1$ of the
complete bipartite graph on $N_0\sqcup[d]$, with $r\geq2$, distinct
$k_1\,,\ldots,k_r\in N_0$ and distinct $i_1\,,\ldots,i_r\in[d]$. Setting $k_{r+1}:=k_1$\,, we
split its arcs into the two alternating sets
$C^+:=\{(k_t\,,i_t): t\in[r]\}$ and $C^-:=\{(k_{t+1}\,,i_t): t\in[r]\}$, and we define the
\emph{cycle sum}
\[
\delta_C(w):=\sum_{(k,i)\in C^+}w_{k,i}-\sum_{(k,i)\in C^-}w_{k,i}
=\sum_{t=1}^r \left(w_{k_t\,,i_t} - w_{k_{t+1}\,,i_t}\right).
\]
Reversing the orientation of $C$ interchanges $C^+$ and $C^-$, and negates $\delta_C$\,.
With $w_{k,i}=\log \bm{a}^k_i$\,, we have $\delta_C(w)=0$ if and only if
$\prod_{t=1}^r \bm{a}^{k_t}_{i_t}/\bm{a}^{k_{t+1}}_{i_t} = 1$.

\begin{theorem}\label{uniquesp}
Let $F(\alpha,s)$ be facet describing. If $G(\alpha)$ contains a simple cycle $C$, then
$\delta_C(w)=0$. Consequently, if $\delta_C(w)\neq0$ for every simple cycle $C$, then, for
every facet-describing $F(\alpha,s)$, the equality graph $G(\alpha)$ is a spanning tree,
which is the unique spanning tree representing $F(\alpha,s)$, and $\mathscr{S}$ is a
triangulation.
\end{theorem}

\begin{proof}
With $u=\log M(\alpha)$ as in \eqref{logrelation}, for each $t \in [r]$ both
$(k_t\,,i_t)$ and $(k_{t+1}\,,i_t)$ belong to $E(\alpha)$, so
$\log\alpha_{i_t}=u_{k_t}+w_{k_t\,,i_t}=u_{k_{t+1}}+w_{k_{t+1}\,,i_t}$\,, that is,
$u_{k_{t+1}}-u_{k_t}=w_{k_t\,,i_t}-w_{k_{t+1}\,,i_t}$\,. Summing over $t\in[r]$, the left-hand
side telescopes to $0$, and the right-hand side is $\delta_C(w)$.

Now suppose that $\delta_C(w)\neq0$ for every simple cycle $C$. Then $G(\alpha)$ is
connected (Theorem \ref{thm:connectedness}) and has no cycle, so it is a spanning tree,
and it is its own unique spanning tree. Any feasible spanning tree $T$ with
$F(\alpha^T,s)=F(\alpha,s)$ satisfies $T\subset E(\alpha)$ by Proposition \ref{prop:sptoF}, so
$T=G(\alpha)$; that is, $G(\alpha)$ is the unique spanning tree representing $F(\alpha,s)$. By Theorem \ref{facetcellbijection},
every maximal cell $C$ of $\mathscr{S}$ has $G(C)$ a spanning tree, so it has $n+d$
vertices and dimension $n+d-1$; hence it is a simplex.
\end{proof}

Let $f(\D)$ denote the number of facets of $\D$, and let $f_{\mathrm{v}}(\D)$ denote the
number of vertical facets of $\D$.

\begin{theorem}\label{thm:facetcount}
Let $m$ be the number of maximal cells of $\mathscr{S}$. Then $f_{\mathrm{v}}(\D) = 2^d m$
and $m \leq \binom{n+d-1}{n}$, with equality if and only if $\mathscr{S}$ is a
triangulation. Therefore,
\[
f(\D) \leq n+1+2^d \binom{n+d-1}{n}.
\]
Equality holds whenever $\delta_C(\log \bm{a})\neq0$ for every simple cycle $C$; in particular,
the bound is attained for all $\bm{a}$ in a dense open subset of
$\mathbb{R}^{N_0\times[d]}_{++}$\,.
\end{theorem}

\begin{proof}
Each vertical facet of $\D$ is described by a unique normalized support inequality
$F(\alpha,s)$ (see the proof of Theorem \ref{thm:completedescrip}), and different pairs
$(\alpha,s)$ give different inequalities. For each of the $2^d$ sign vectors $s$,
Theorem \ref{facetcellbijection} gives exactly $m$ such facets, so $f_{\mathrm{v}}(\D)=2^dm$.

We use normalized volume $\Vol_{\mathbb{Z}}$ with respect to the lattice
$\aff(Q)\cap\mathbb{Z}^{N_0\times[d]}$. The vertices of $Q$ are the columns of the
incidence matrix of the complete bipartite graph on $N_0\sqcup[d]$, which is totally
unimodular. Consequently, every $(n+d-1)$-simplex with vertices in $V(Q)$ has
normalized volume $1$, and $\Vol_{\mathbb{Z}}(Q)=\binom{n+d-1}{n}$; see
\cite[Section 6.2]{DeloeraRambuSantos2010}. Each maximal cell $C$ can be triangulated
using its own vertices into some number $t_C\geq1$ of simplices, so
$\Vol_{\mathbb{Z}}(C)=t_C$\,, and $t_C=1$ if and only if $C$ is a simplex. Because the
maximal cells cover $Q$ and have pairwise disjoint interiors,
\[
m\ \leq\ \sum_{C}t_C\ =\ \sum_{C}\Vol_{\mathbb{Z}}(C)\ =\ \Vol_{\mathbb{Z}}(Q)\ =\ \binom{n+d-1}{n},
\]
with equality if and only if every maximal cell is a simplex. Together with the $n+1$
non-vertical facets (see \cite[Theorems 4 and 5]{QL_DAM2025}), this gives the bound on
$f(\D)$.

Finally, if $\delta_C(\log \bm{a})\neq0$ for every simple cycle $C$, then $\mathscr{S}$ is a
triangulation by Theorem \ref{uniquesp}, so equality holds. There are finitely many
simple cycles, and each $\delta_C$ is a nonzero linear function of $\log \bm{a}$, so this
condition holds on the complement of finitely many hypersurfaces in
$\mathbb{R}^{N_0\times[d]}_{++}$\,, which is dense and open.
\end{proof}

We next study how feasible spanning trees change with the heights. All of the results above apply to arbitrary heights $w$, by setting $\bm{a}^k_i:=e^{w_{k,i}}$\,. For a spanning tree
$T$ on $N_0\sqcup[d]$ and heights $w$, let $M^T(w)$ and $\alpha^T(w)$ be as in
Section \ref{sec:spanning}, computed with $\bm{a}^k_i=e^{w_{k,i}}$\,. Root $T$ at $0$. Each node
$i\in[d]$ has a parent $k\in N_0$ in $T$, and the path from $0$ to every child $k'$ of $i$
passes through $k-i-k'$, so the definition of $M^T$ gives
$M^T_k \bm{a}^k_i=M^T_{k'}\bm{a}^{k'}_i$\,. Therefore,
\[
\tau^T_i(w):=M^T_k(w)\,e^{w_{k,i}}\,,\quad\mbox{for any } (k,i)\in T,
\]
is well defined. We define the \emph{slack}
\[
\sigma^T_{l,i}(w):=\log M^T_l(w)+w_{l,i}-\log\tau^T_i(w)
= \log\frac{M^T_l(w)\,e^{w_{l,i}}}{\tau^T_i(w)},
  \qquad (l,i)\in N_0\times[d].
\]
By construction, $\sigma^T_{l,i}(w)=0$ for $(l,i)\in T$. Because $\alpha^T_i(w)=\min_{l\in N_0}M^T_l(w)e^{w_{l,i}}$ and the minimum includes the value $\tau^T_i(w)$, we always have $\alpha^T(w)\le\tau^T(w)$. Hence the following are equivalent:
(i) $T$ is feasible at $w$; (ii) $\sigma^T_{l,i}(w)\ge 0$ for every arc
$(l,i)\notin T$; (iii) $\alpha^T(w)=\tau^T(w)$. For $e\notin T$, the graph $T\cup\{e\}$ contains a unique simple cycle, the \emph{fundamental cycle} $C(T,e)$, which we orient so that $e\in C(T,e)^+$.

\begin{lemma}\label{lem:slack}
For every spanning tree $T$ on $N_0\sqcup[d]$, every $e\notin T$, and all heights $w$,
we have $\sigma^T_e(w)=\delta_{C(T,e)}(w)$. Consequently, $T$ is feasible at $w$ if and
only if $\delta_{C(T,e)}(w)\geq0$ for every $e\notin T$.
\end{lemma}

\begin{proof}
Write $C(T,e): k_1 - i_1 - k_2 - \cdots - k_r - i_r - k_1$ with $e=(k_1\,,i_1)\in C(T,e)^+$, and
let $u:=\log M^T(w)$. Every arc of $C(T,e)$ other than $e$ lies in $T$, so, by the
definition of $\tau^T$,
\begin{align*}
\sigma^T_e(w)&=(u_{k_1}+w_{k_1\,,i_1})-(u_{k_2}+w_{k_2\,,i_1}),\\
0&=(u_{k_t}+w_{k_t\,,i_t})-(u_{k_{t+1}}+w_{k_{t+1}\,,i_t}),\quad t=2,\ldots,r.
\end{align*}
Summing these $r$ equations, the terms in $u$ telescope, and we obtain
$\sigma^T_e(w)=\delta_{C(T,e)}(w)$.
\end{proof}

\begin{theorem}\label{thm:flip}
Let $w^0$ be heights, and let $C^0$ be a maximal cell of $\mathscr{S}(w^0)$, with
associated facet $F(\alpha^0,s)$ (Theorem \ref{facetcellbijection}), such that $G(C^0)$
contains exactly one cycle $C: k_1 - i_1 - k_2 - \cdots - k_r - i_r - k_1$\,. Then
$|E(C^0)| = n+d+1$, and the spanning trees of $G(C^0)$ are the $2r$ trees
$T_e:=E(C^0)\setminus\{e\}$, for $e\in C^+\cup C^-$. Moreover, the following hold.
\begin{enumerate}
    \item At $w=w^0$, every $T_e$ is feasible and $F(\alpha^{T_e}(w^0),s)=F(\alpha^0,s)$.
    \item There is a neighborhood $U$ of $w^0$ such that, for every $w \in U$,
    \[
    T_e \text{ is feasible at } w \iff
    \begin{cases}
    \delta_C(w) \geq 0\,, & \mbox{if } e \in C^+; \\
    \delta_C(w) \leq 0\,, & \mbox{if } e \in C^-.
    \end{cases}
    \]
    \item For every $w\in U$ with $\delta_C(w)>0$, the $r$ simplices
    $C_{T_e}=\conv\{v_{k,i}:(k,i)\in T_e\}$, for $e\in C^+$, are maximal cells of
    $\mathscr{S}(w)$, and they cover $C^0$. The same holds with $C^-$ in place of $C^+$
    when $\delta_C(w)<0$.
    \item For every $w\in U$ with $\delta_C(w)>0$ and each $s\in\{\pm1\}^d$, the $r$ facets
    $F(\alpha^{T_e}(w),s)$, for $e\in C^+$, are pairwise distinct, and their coefficients
    converge to those of $F(\alpha^0,s)$ as $w\to w^0$. The same holds with $C^-$ in place
    of $C^+$ when $\delta_C(w)<0$.
\end{enumerate}
\end{theorem}

\begin{proof}
By Theorem \ref{facetcellbijection}, $E(C^0)=E(\alpha^0)$, and $G(C^0)$ is connected
with node set $N_0\sqcup[d]$. A connected graph on $n+d+1$ nodes with exactly one cycle
has $n+d+1$ arcs, and its spanning trees are obtained by deleting exactly one arc of
the cycle. This gives the first two claims. Note also that $\delta_C(w^0)=0$, by
Theorem \ref{uniquesp}.

\emph{(1)}. Each $T_e$ is a spanning tree of $G(\alpha^0)$, so, by
Proposition \ref{sprecover}, it is feasible at $w^0$, with $\alpha^{T_e}(w^0)=\alpha^0$
and $M^{T_e}(w^0)=M(\alpha^0)$.

\emph{(2)}. Fix $e\in C^+\cup C^-$. The arcs not in $T_e$ are $e$ and the arcs
$f\notin E(\alpha^0)$. For such $f=(l,i)$, part (1) gives
$\sigma^{T_e}_f(w^0)=\log\bigl(M_l(\alpha^0)\bm{a}^l_i/\alpha^0_i\bigr)>0$, because
$f\notin E(\alpha^0)$. Because $M^{T_e}(w)$ and $\tau^{T_e}(w)$ are continuous (indeed,
products of exponentials of entries of $w$), there is a neighborhood $U_e$ of $w^0$ on
which $\sigma^{T_e}_f(w)>0$ for all $f\notin E(\alpha^0)$. Let $U:=\bigcap_e U_e$\,, a
finite intersection. For $w\in U$, Lemma \ref{lem:slack} shows that $T_e$ is feasible
at $w$ if and only if $\sigma^{T_e}_e(w)\geq0$. The fundamental cycle of $e$ in $T_e$
is $C$ itself, oriented so that $e$ is in its positive part; this is the given
orientation of $C$ if $e\in C^+$, and the reverse orientation if $e\in C^-$. By
Lemma \ref{lem:slack}, $\sigma^{T_e}_e(w)$ equals $\delta_C(w)$ if $e\in C^+$ and
$-\delta_C(w)$ if $e\in C^-$.

\emph{(3)}. Let $w\in U$ with $\delta_C(w)>0$, and let $e\in C^+$. By (2), $T_e$ is
feasible at $w$, and every arc not in $T_e$ has positive slack, so
$E(\alpha^{T_e}(w))=T_e$\,. By Proposition \ref{prop:sptoF} and
Theorem \ref{facetcellbijection}, $C_{T_e}$ is a maximal cell of $\mathscr{S}(w)$, and it
is a simplex, because $|T_e|=n+d$. To see that these cells cover $C^0$, note that
\[
\sum_{(k,i)\in C^+}v_{k,i}-\sum_{(k,i)\in C^-}v_{k,i}
=\sum_{t=1}^r\left((p^{k_t},q^{i_t})-(p^{k_{t+1}},q^{i_t})\right)=0,
\]
and the coefficients of this relation sum to zero. Let $x\in C^0$, and write
$x=\sum_{(k,i)\in E(C^0)}\lambda_{k,i}v_{k,i}$ with $\lambda\geq0$ and
$\sum\lambda_{k,i}=1$. Let $\theta:=\min\{\lambda_{k,i}:(k,i)\in C^+\}$, attained at
some $e\in C^+$. Subtracting $\theta$ times the relation above, that is, replacing
$\lambda_{k,i}$ by $\lambda_{k,i}-\theta$ on $C^+$ and by $\lambda_{k,i}+\theta$ on
$C^-$, gives a convex combination of the same point with $\lambda_e=0$, so
$x\in C_{T_e}$\,. The case $\delta_C(w)<0$ is symmetric, adding $\theta$ times the
relation with $\theta$ minimized over $C^-$.

\emph{(4)}. For distinct $e,e'\in C^+$, the maximal cells $C_{T_e}$ and $C_{T_{e'}}$ are
distinct, so, by Theorem \ref{facetcellbijection}, the facets
$F(\alpha^{T_e}(w),s)$ and $F(\alpha^{T_{e'}}(w),s)$ are distinct. The coefficients
$\alpha^{T_e}(w)$ and $M^{T_e}(w)$ depend continuously on $w$, and by (1), they equal
$\alpha^0$ and $M(\alpha^0)$ at $w=w^0$. The case $\delta_C(w)<0$ is symmetric.
\end{proof}

\begin{corollary}\label{cor:chambers}
Let $\mathscr{C}$ be the set of simple cycles, and for $C\in\mathscr{C}$, let
$\mathcal{H}_C := \{w \in \mathbb{R}^{N_0 \times [d]} : \delta_C(w) = 0\}$, a hyperplane.
Let $\mathcal{H} := \bigcup_{C \in \mathscr{C}} \mathcal{H}_C$\,, and for heights $w$,
let $\mathcal{T}(w)$ be the set of spanning trees that are feasible at $w$. If $w$ and $w'$
lie in the same connected component $\mathcal{R}$ of
$\mathbb{R}^{N_0 \times [d]} \setminus \mathcal{H}$, then
$\mathcal{T}(w) = \mathcal{T}(w')$ and $\mathscr{S}(w) = \mathscr{S}(w')$, and this
subdivision is a triangulation.
\end{corollary}

\begin{proof}
For each $C\in\mathscr{C}$, the continuous function $\delta_C$ has no zero on the
connected set $\mathcal{R}$, so it has constant sign there. By Lemma \ref{lem:slack}, a
spanning tree $T$ is feasible at $w\in\mathcal{R}$ if and only if
$\delta_{C(T,e)}(w)>0$ for every $e\notin T$, which depends only on these signs. Hence
$\mathcal{T}(w)=\mathcal{T}(w')$. By Theorem \ref{uniquesp}, $\mathscr{S}(w)$ is a
triangulation, and the graph of each of its maximal cells is a spanning tree. By
Theorem \ref{facetcellbijection} and Corollary \ref{cor:treefacets}, the maximal cells
of $\mathscr{S}(w)$ are therefore exactly the simplices
$\conv\{v_{k,i}:(k,i)\in T\}$, for $T\in\mathcal{T}(w)$. The same holds at $w'$, so
$\mathscr{S}(w)=\mathscr{S}(w')$.
\end{proof}

\begin{remark}
     Although the feasible spanning trees and the regular triangulation are constant on $\mathcal{R}$, the numerical coefficients $\alpha^T,M^T$ are not: for $T \in \mathcal{T}(w)$, in general $(\alpha^T(w),M^T(w)) \neq (\alpha^T(w'),M^T(w'))$.
\end{remark}

\begin{example}\label{ex:square-perturbation}
We return to Example~\ref{ex:spcalc} and perturb one parameter.
Let $n=1$, $d=2$, and, for $t\in\mathbb{R}$, set
\[
  (\bm{a}_1^0\,,\bm{a}_2^0)=(e^t,3),\qquad
  (\bm{a}_1^1\,,\bm{a}_2^1)=(3,9).
\]
At $t=0$, the normalized coefficients are
$\alpha=(1,3)$ and $M_1=1/3$\,. All four arcs are tight, so the
equality graph is $K_{2,2}$ and its four spanning trees represent
the same facet. For $s=(1,1)$, this facet is
\[
  x_1+3x_2+\frac23z_1\le1.
\]
Choose the cycle $C: 0-1-1-2-0$ (alternating between $N_0$ and $[2]$), so that
\[
  C^+=\{(0,1),(1,2)\},\qquad
  C^-=\{(1,1),(0,2)\}.
\]
Its cycle sum is
\[
  \delta_C(w)=w_{0,1}-w_{1,1}+w_{1,2}-w_{0,2}
             =t-\log3+\log9-\log3=t.
\]

We draw $Q$ in lattice coordinates, with
$v_{0,1}=(0,0)$\,, $v_{1,1}=(1,0)$\,,
$v_{0,2}=(0,1)$\,, and $v_{1,2}=(1,1)$\,.
At $t=0$, the four lifted vertices lie on one affine plane, so
the only maximal cell is the whole square.
For $t<0$, the subdivision uses the diagonal
$\conv\{v_{0,1}\,,v_{1,2}\}$; for $t>0$, it uses
$\conv\{v_{1,1}\,,v_{0,2}\}$\,. Figure~\ref{fig:square-perturbation}
shows these three subdivisions.

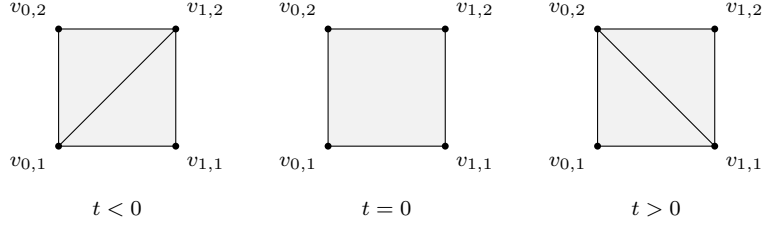
\begin{figure}[htbp]
\centering
\begin{tikzpicture}[x=1.55cm,y=1.55cm,
  every node/.style={font=\scriptsize}]
  \begin{scope}
    \fill[black!5] (0,0) rectangle (1,1);
    \draw (0,0) rectangle (1,1);
    \draw (0,0)--(1,1);
    \foreach \x/\y in {0/0,1/0,0/1,1/1}
      \fill (\x,\y) circle (1.3pt);
    \node[below left] at (0,0) {$v_{0,1}$};
    \node[below right] at (1,0) {$v_{1,1}$};
    \node[above left] at (0,1) {$v_{0,2}$};
    \node[above right] at (1,1) {$v_{1,2}$};
    \node[below] at (.5,-.38) {$t<0$};
  \end{scope}
  \begin{scope}[shift={(2.3,0)}]
    \fill[black!5] (0,0) rectangle (1,1);
    \draw (0,0) rectangle (1,1);
    \foreach \x/\y in {0/0,1/0,0/1,1/1}
      \fill (\x,\y) circle (1.3pt);
    \node[below left] at (0,0) {$v_{0,1}$};
    \node[below right] at (1,0) {$v_{1,1}$};
    \node[above left] at (0,1) {$v_{0,2}$};
    \node[above right] at (1,1) {$v_{1,2}$};
    \node[below] at (.5,-.38) {$t=0$};
  \end{scope}
  \begin{scope}[shift={(4.6,0)}]
    \fill[black!5] (0,0) rectangle (1,1);
    \draw (0,0) rectangle (1,1);
    \draw (1,0)--(0,1);
    \foreach \x/\y in {0/0,1/0,0/1,1/1}
      \fill (\x,\y) circle (1.3pt);
    \node[below left] at (0,0) {$v_{0,1}$};
    \node[below right] at (1,0) {$v_{1,1}$};
    \node[above left] at (0,1) {$v_{0,2}$};
    \node[above right] at (1,1) {$v_{1,2}$};
    \node[below] at (.5,-.38) {$t>0$};
  \end{scope}
\end{tikzpicture}
\caption{The subdivisions of the square in
Example~\ref{ex:square-perturbation}.}
\label{fig:square-perturbation}
\end{figure}

Write $E_0=\{0,1\}\times\{1,2\}$ and
$T_{k,i}:=E_0\setminus\{(k,i)\}$\,. The feasible trees and their
normalized coefficients are as follows; in every row, $M_0=1$\,.
\begin{center}
\begin{tabular}{@{}cccc@{}}
\toprule
Parameter & Tree & $(\alpha_1\,,\alpha_2)$ & $M_1$\\
\midrule
$t>0$ & $T_{0,1}$ & $(1,3)$ & $1/3$\\
$t>0$ & $T_{1,2}$ & $(e^t,3)$ & $e^t/3$\\
$t<0$ & $T_{1,1}$ & $(e^t,3)$ & $1/3$\\
$t<0$ & $T_{0,2}$ & $(e^t,3e^t)$ & $e^t/3$\\
\bottomrule
\end{tabular}
\end{center}
For example, the equalities on $T_{1,2}$ give
$\alpha_1=e^t=3M_1$ and $\alpha_2=3$\,.
The omitted arc $(1,2)$ requires that $3 = \alpha_2 \leq M_1 \bm{a}^1_2 = 9M_1=3e^t$\,, which holds
exactly when $t\ge0$. The other rows follow in the same way.
For $s=(1,1)$, each row gives the facet inequality
\[
  \alpha_1x_1+\alpha_2x_2+(1-M_1)z_1\le1.
\]
In agreement with Theorem \ref{thm:flip}, for $t>0$ the feasible trees are $T_e$ with $e\in C^+$, and for $t<0$ they are $T_e$ with $e\in C^-$. On either side of zero, the two rows give distinct facets. As
$t\to0$, all four coefficient vectors converge to
$(\alpha_1\,,\alpha_2\,,M_1)=(1,3,1/3)$\,.
\end{example}

\section{Separation}\label{sec:sep}

Membership in $\D$ by itself is easy to decide: writing each $\Po_k$ as
$\{x: \exists\, y\in\mathbb{R}^d,\ -y\le x\le y,\ \sum_{i\in[d]}\bm{a}^k_iy_i\le1\}$, the
disjunctive-programming formulation of \cite{BalasDP} expresses $\D$ as the projection of
a polytope with $O(nd)$ variables and constraints, so membership reduces to a single
linear program of polynomial size. What that approach does not provide is a facet of
$\D$. We now give a polynomial-time algorithm that either certifies membership or
returns a violated \emph{facet-describing} inequality of $\D$, which, when vertical, is
a most violated one among the normalized support inequalities. Such inequalities can
be used as cutting planes in larger mixed-integer models in which $\D$ appears as a
substructure. Note that the algorithm is polynomial in both $n$ and $d$, whereas $\D$
can have as many as $n+1+2^d\binom{n+d-1}{n}$ facets (Theorem \ref{thm:facetcount}).

Let $(\bar x,\bar z)\in\mathbb{R}^{d+n}$, set $c_i:=|\bar x_i|$ for $i\in[d]$, and define
the polyhedron
\begin{equation}\label{eq:separation-polyhedron}
\mathcal L:=\left\{(\alpha,M)\in\mathbb{R}^d\times\mathbb{R}^{N_0}:
\begin{array}{ll}
\alpha_i\le M_k \bm{a}_i^k & (k\in N_0\,,\ i\in[d]),\\
M_0=1,\quad \alpha_i\ge0 & (i\in[d]),\\
M_k\ge0 & (k\in N)
\end{array}\right\}
\end{equation}
and the linear function
\[
\phi(\alpha,M):=
\sum_{i\in[d]}c_i\alpha_i+
\sum_{k\in N}(1-M_k)\bar z_k-1 .
\]
For every normalized support inequality $F(\alpha,s)$, the point $(\alpha,M(\alpha))$
lies in $\mathcal L$, and the violation of $F(\alpha,s)$ at $(\bar x,\bar z)$ is at most
$\phi(\alpha,M(\alpha))$, with equality when $s_i\bar x_i=|\bar x_i|$ for all $i$.

\medskip
\noindent {\bf Algorithm.}
\begin{enumerate}[label=\arabic*.]
\item If $(\bar x,\bar z)$ violates one of the non-vertical inequalities $z_k\geq0$, for
$k\in N$, or $\sum_{k\in N}z_k\leq1$, return it.
\item \emph{Maximize the violation.} Compute
$\theta:=\max\{\phi(\alpha,M):(\alpha,M)\in\mathcal L\}$ and a basic optimal solution
$(\alpha^1,M^1)$. If $\theta\le0$, return ``$(\bar x,\bar z)\in\D$''.
\item \emph{Saturate.} Set $\alpha_i:=\min_{k\in N_0}M^1_k \bm{a}^k_i$ for $i\in[d]$, and then
$M_0:=1$ and $M_k:=\max_{i\in[d]}\alpha_i/\bm{a}^k_i$ for $k\in N$.
\item \emph{Connect.} While $G(\alpha)$ is not connected: let $C_0$ be the node set of
the connected component of $G(\alpha)$ containing the node $0$, let
\[
\lambda:=\min\left\{\frac{M_k \bm{a}^k_i}{\alpha_i}~:~k\in C_0\cap N_0\,,\ i\in[d]\setminus C_0\right\},
\]
and multiply $\alpha_i$ by $\lambda$ for every $i\in[d]\setminus C_0$ and $M_k$ by $\lambda$ for
every $k\in N_0\setminus C_0$\,.
\item Let $s_i:=1$ if $\bar x_i\ge0$ and $s_i:=-1$ otherwise, for $i\in[d]$, and return
$F(\alpha,s)$.
\end{enumerate}

Informally, Step 2 finds a most violated normalized support inequality, Step 3 makes
each coefficient as large (for $\alpha$) or as small (for $M$) as the others permit, and
Step 4 moves along directions in which the violation is constant until the equality
graph becomes connected, so that the output is facet describing.

\begin{theorem}\label{thm:separation}
Suppose that the data $\bm{a}_i^k$ and the point $(\bar x,\bar z)$ are rational. In time
polynomial in the binary encoding length of the input, the algorithm either certifies
that $(\bar x,\bar z)\in\D$ or returns a facet-describing inequality of $\D$ violated by
$(\bar x,\bar z)$. In the latter case, if the returned inequality is vertical, then it
is a normalized support inequality $F(\alpha,s)$, and its violation is maximum
among all normalized support inequalities.
\end{theorem}

\begin{proof}
The non-vertical inequalities describe facets of $\D$ (see \cite[Theorems 4 and 5]{QL_DAM2025}),
so we may assume that $\bar z\ge0$ and $\sum_{k\in N}\bar z_k\le1$.

\emph{Step 2 is well defined.} The polyhedron $\mathcal L$ is nonempty, because it
contains the point with $\alpha=0$, $M_0=1$, and $M_k=0$ for $k\in N$. Every feasible
point satisfies $0\le\alpha_i\le \bm{a}_i^0$\,, and, because $\bar z\ge0$,
$\phi(\alpha,M)\le\sum_{i\in[d]}c_i \bm{a}_i^0+\sum_{k\in N}\bar z_k-1$. Moreover, $\mathcal L$
contains no line, because all of its coordinates are nonnegative. Hence the maximum is
attained at a vertex of $\mathcal L$.

\emph{The case $\theta\le0$.} By the remark preceding the algorithm, every normalized
support inequality, and hence every vertical facet inequality, is satisfied by
$(\bar x,\bar z)$. The non-vertical inequalities are satisfied as well, so
$(\bar x,\bar z)\in\D$ by Theorem~\ref{thm:completedescrip}.

From now on, suppose that $\theta>0$. We say that $(\alpha,M)\in\mathcal L$ is
\emph{saturated} if
\begin{equation}\label{eq:separation-saturation}
\alpha_i=\min_{k\in N_0}M_k \bm{a}_i^k\quad(i\in[d]),
\qquad
M_k = \max_{i\in[d]}\frac{\alpha_i}{\bm{a}_i^k}\quad(k\in N_0).
\end{equation}
For a saturated point, $M=M(\alpha)$, and the constraints $\alpha_i\le M_k \bm{a}^k_i$ that are
tight are exactly those with $(k,i)\in E(\alpha)$.

\emph{Step 3 gives an optimal saturated point.} Let $\alpha'$ and $M'$
denote the vectors computed in Step~3.

\emph{Raising $\alpha$.} Because $(\alpha^1,M^1)\in\mathcal L$, we have
$\alpha^1_i\le M^1_k \bm{a}^k_i$ for all $k$, so
$\alpha'_i=\min_kM^1_k \bm{a}^k_i\ge\alpha^1_i$\,. 
Thus $\alpha'\ge\alpha^1\ge0$, so $(\alpha',M^1)\in\mathcal L$, and, because $c\ge0$,
$\phi(\alpha',M^1)-\phi(\alpha^1,M^1)=\sum_i c_i(\alpha'_i-\alpha^1_i)\ge0$;
hence $(\alpha',M^1)$ is optimal.

\emph{Lowering $M$.} For $k\in N$, taking $k$ in the minimum defining
$\alpha'_i$ gives $\alpha'_i\le M^1_k \bm{a}^k_i$\,, so $M'_k =\max_{i\in[d]}\frac{\alpha'_i}{\bm{a}_i^k}\leq M^1_k$\,. The
point $(\alpha',M')$ is feasible: the constraints with $k\in N$ hold by
the definition of $M'_k$\,, and those with $k=0$ hold because
$\alpha'_i\le M^1_0 \bm{a}^0_i=\bm{a}^0_i$\,. Because $\bar z\geq0$ and $\phi\left(\alpha^{\prime}, M^{\prime}\right)-\phi\left(\alpha^{\prime}, M^1\right)=\sum_{k \in N}\left(M_k^1-M_k^{\prime}\right) \bar{z}_k \geq 0$, it is optimal as
well. In particular, the constraints with $k=0$ give $M_0(\alpha') = \max_i\alpha'_i/\bm{a}^0_i \leq 1$.

\emph{Normalization and the second identity.} We claim that $M_0(\alpha')=1$. Otherwise every
constraint $\alpha'_i\le \bm{a}^0_i$ is strict. Then, for sufficiently small
$\varepsilon>0$, multiplying $\alpha'$ and all $M'_k$ with $k\in N$ by
$1+\varepsilon$, while keeping $M_0=1$, gives a feasible point: the
constraints with $k\in N$ are scaled on both sides, and those with
$k=0$ survive by strictness. Because $\phi(\alpha',M')=\theta$, we have
$\sum_ic_i\alpha'_i-\sum_{k\in N}M'_k\bar z_k
=\theta+1-\sum_{k\in N}\bar z_k$\,, so this change increases $\phi$ by
\[
  \varepsilon\Bigl(\theta+1-\sum_{k\in N}\bar z_k\Bigr)>0,
\]
because $\theta>0$ and $\sum_{k\in N}\bar z_k\le1$. This contradicts
optimality. Hence $M'_k=\max_i\alpha'_i/\bm{a}^k_i$ for every $k\in N_0$\,,
which is the second identity in~\eqref{eq:separation-saturation}.

\emph{The first identity.} Fix $i\in[d]$ and choose $k^*$ attaining
$\alpha'_i=M^1_{k^*}\bm{a}^{k^*}_i$\,. Then
\[
  M^1_{k^*}\;\ge\;M'_{k^*}\;\ge\;\alpha'_i/\bm{a}^{k^*}_i\;=\;M^1_{k^*}\,,
\]
by $M'\le M^1$ (if $k^*=0$, both sides equal $1$) and by the second
identity. Hence $M'_{k^*}=M^1_{k^*}$\,, and
$\alpha'_i=M'_{k^*}\bm{a}^{k^*}_i\ge\min_kM'_k \bm{a}^k_i\ge\alpha'_i$\,, the last
inequality by feasibility. Thus $(\alpha',M')$ is an optimal saturated
point.

\emph{Positivity.} Because $M_0(\alpha')=1$, some $\alpha'_i>0$, so every
$M'_k=\max_i\alpha'_i/\bm{a}^k_i>0$, and then the first identity gives
$\alpha'>0$.

\emph{Step~4 preserves optimality and saturation, and terminates.}
Consider an iteration of Step~4 at an optimal saturated point $(\alpha,M)$
with $\alpha>0$. Because Step~4 is executed, $G(\alpha)$ is disconnected.
For $\mu>0$, let $(\alpha(\mu),M(\mu))$ be obtained by multiplying
$\alpha_i$ for $i\in[d]\setminus C_0$ and $M_k$ for $k\in N_0\setminus C_0$
by $\mu$. Because $0\in C_0$\,, the normalization $M_0=1$ is unchanged.

\emph{Feasibility.} Consider a constraint $\alpha_i\le M_k \bm{a}^k_i$\,. If $k$
and $i$ both lie in $C_0$\,, or both lie outside, the constraint is
unaffected or scaled on both sides. Otherwise $(k,i)\notin E(\alpha)$,
because an arc of $E(\alpha)$ cannot join $C_0$ to another component, so the
constraint is strict at $(\alpha,M)$. If $k\notin C_0$ and $i\in C_0$\,, its
right-hand side is multiplied by $\mu$, so it holds for all $\mu\ge1$;
strictness gives $\alpha_i/(M_k \bm{a}^k_i)<1$, so it also holds for $\mu$
sufficiently close to $1$ from below. If $k\in C_0$ and $i\notin C_0$\,, it
holds if and only if $\mu\le M_k \bm{a}^k_i/\alpha_i$\,, and this bound exceeds
$1$ by strictness. Pairs of the latter kind exist: $0\in C_0$\,, and every
component other than $C_0$ contains a node of $[d]$, because by saturation
every $k\in N_0$ is incident to an arc of $E(\alpha)$. Hence $\lambda>1$,
and $(\alpha(\mu),M(\mu))\in\mathcal{L}$ for all
$\mu\in[1-\varepsilon,\lambda]$ for some $0<\varepsilon<1$.

\emph{Optimality.} Each coordinate of $(\alpha(\mu),M(\mu))$ is either
constant or proportional to $\mu$, so $\phi(\alpha(\mu),M(\mu))$ is affine
in $\mu$. Because this segment lies in $\mathcal{L}$ and $(\alpha,M)$ is
optimal, $\phi$ attains its maximum over $[1-\varepsilon,\lambda]$ at the
interior point $\mu=1$. Hence it is constant, and $(\alpha(\lambda),M(\lambda))$
is optimal.

\emph{Saturation.} Every arc of $E(\alpha)$ joins two nodes on the same
side of $C_0$\,, so its constraint is scaled on both sides and remains tight.
Together with feasibility at $\mu=\lambda$, this gives equality in both
identities of~\eqref{eq:separation-saturation}, so the new point is
saturated, and $\alpha(\lambda)>0$.

\emph{Progress.} Because $E(\alpha)\subset E(\alpha(\lambda))$, and an arc
$(k,i)$ with $k\in C_0$ and $i\notin C_0$ attaining $\lambda$ becomes tight,
the component containing $0$ contains $C_0$ together with the component of
$i$. Thus it gains at least one node in each iteration. Because $G(\alpha)$
has $n+d+1$ nodes, Step~4 terminates after at most $n+d$ iterations.

\emph{The output is facet describing and maximally violated.} At the end of Step 4, the
point $(\alpha,M)$ is optimal and saturated, with $\alpha>0$, $M=M(\alpha)$, and $G(\alpha)$
connected. So $F(\alpha,s)$ is a normalized support inequality, and it is facet describing
by Theorem~\ref{thm:connectedness}. Because $s_i\bar x_i=|\bar x_i|$, its violation is
$\phi(\alpha,M)=\theta>0$, and, by the remark preceding the algorithm, every normalized
support inequality has violation at most $\theta$.

\emph{Running time.} Step 2 is one linear program with $d+n+1$ variables and $O(nd)$
constraints, and a basic optimal solution can be computed in polynomial time, with
polynomial encoding length. Steps 3 and 4 use $O((n+d)\,nd)$ arithmetic operations. It
remains to bound the encoding lengths in Step 4. During Step 4, the coordinates indexed
by the nodes of $C_0$ are determined by a spanning tree of the component, through
$M_0=1$ and the relations $\alpha_i=M_k \bm{a}^k_i$ on its arcs, exactly as in the definition
of $M^T$ in Section \ref{sec:spanning}; so each is a product of at most $n+d$ ratios of
entries of $\bm{a}$, of polynomial encoding length. Each coordinate outside $C_0$ equals its
value after Step 3 times the product of the factors $\lambda$ used so far, and, after an
iteration in which $\lambda=M_k \bm{a}^k_i/\alpha_i$\,, this product equals the ratio of the
value $M_k \bm{a}^k_i$ (indexed by $C_0$) to the value of $\alpha_i$ after Step 3. By induction,
all numbers have polynomial encoding length.
\end{proof}

\section{Computational experiments}\label{sec:experiments}

We measure how much of the relaxation gap is closed by the facets of $\D$ that are not
pairwise inequalities, extending the experiments of
\cite[Section 5]{qu2026disjunctionconvexhullsgeneralized}.

\medskip
\noindent {\bf Instances.}
For each $(d,n)$ with $3\leq d\leq12$ and $2\leq n\leq20$, we generate five random
instances, $950$ in total. For $k\in N_0$\,, the entries $\bm{a}^k_i$ are drawn independently
and uniformly from $[1,4]$ (so the ratios $\bm{a}^{k_1}_i/\bm{a}^{k_2}_i$ are distinct with
probability one), the center $t^k$ is drawn uniformly from the unit sphere in
$\mathbb{R}^d$, and
$\Po_k:=\{x\in\mathbb{R}^d:\sum_{i\in[d]}\bm{a}^k_i|x_i-t^k_i|\leq1\}$; by
Proposition \ref{prop:translations}, our results apply to these translated polytopes. We
draw $c\in\mathbb{Z}^d$ uniformly from $[-d,d]^d$ (with at least two nonzero entries),
let $v_k:=\max\{c^\top x:x\in\Po_k\}=c^\top t^k+\max_{i\in[d]}|c_i|/\bm{a}^k_i$\,, and set
$g_k:=v_0-v_k$ for $k\in N$, so that the objectives $g_k+c^\top x$ of the $n+1$ modes
compete. As side constraints, we trim from each $\Po_k$ its best vertex and the reflection
of that vertex: with $i_k\in\argmax_{i\in[d]}|c_i|/\bm{a}^k_i$ and $\beta_k:=1/(2 \bm{a}^k_{i_k})$, let
$\Po^+_k:=\Po_k\cap\{x:|x_{i_k}-t^k_{i_k}|\leq\beta_k\}$. The target value is
\[
\mathfrak{z}(\mathcal{P}):=\max_{k\in N_0}\ \bigl\{g_k+\max\{c^\top x:x\in\Po^+_k\}\bigr\},
\]
with $g_0:=0$, which is available in closed form.

\medskip
\noindent {\bf Relaxations.}
We compare the optimal values of $\max\{c^\top x+g^\top z\}$ over the following
relaxations, each of which includes the non-vertical inequalities $z_k\ge0$ and
$\sum_{k\in N}z_k\le1$.
\begin{itemize}
\item $\mathcal{L}$: the pairwise inequalities with $r^{k_1\,,k_2}_j$ equal to
$\min_i r^{k_1\,,k_2}_i$ or $\max_i r^{k_1\,,k_2}_i$\,, that is, the full optimal big-M
liftings of the facet-describing inequalities of the $\Po_k$ (see \cite{qu2026disjunctionconvexhullsgeneralized});
\item $\mathcal{F}$: all pairwise inequalities (Section \ref{sec:prior});
\item $\mathcal{D}$: the complete description of Theorem \ref{thm:completedescrip}.
\end{itemize}
We also consider $\mathcal{L}^{\mathrm s}$, $\mathcal{F}^{\mathrm s}$ and
$\mathcal{D}^{\mathrm s}$, obtained by adding lifted versions of the trimming constraints.
For $k,l\in N_0$\,, let $\gamma_{k,l}:=\max\{|x_{i_k}-t^k_{i_k}|:x\in\Po^+_l\}$, which
equals $|t^l_{i_k}-t^k_{i_k}|+\rho_{l}$\,, where $\rho_l:=1/\bm{a}^l_{i_k}$\,, halved if
$i_k=i_l$\,; in particular, $\gamma_{k,k}=\beta_k$\,. The lifted constraint
\[
|x_{i_k}-t^k_{i_k}|\ \leq\ \gamma_{k,0}+\sum_{l\in N}\left(\gamma_{k,l}-\gamma_{k,0}\right)z_l
\]
is valid for $\conv\bigcup_{l\in N_0}\bar{\Po}^+_l$ and reduces to the trimming
constraint when $z=\mathbf{e}_k$\,. Following \cite{qu2026disjunctionconvexhullsgeneralized},
we report, for each relaxation $X$, the fraction of the gap left by $\mathcal{L}$ that $X$
closes,
\[
1-\frac{\mathfrak{z}(X)-\mathfrak{z}(\mathcal{P})}{\mathfrak{z}(\mathcal{L})-\mathfrak{z}(\mathcal{P})}\,.
\]

\medskip
\noindent {\bf Implementation.}
All inequalities are written in the original (translated) coordinates using the affine
map of Proposition \ref{prop:translations}, and absolute values are handled with
auxiliary variables $w_i\geq|x_i-t^0_i-\sum_{k\in N}z_k(t^k_i-t^0_i)|$, so that each
inequality $\sum_{i}\alpha_i w_i+\sum_{k\in N}(1-M_k)z_k\leq1$ represents all $2^d$ sign
vectors at once. The relaxations $\mathcal{D}$ and $\mathcal{D}^{\mathrm s}$ are computed
by cutting planes: starting from $\mathcal{F}$ (respectively $\mathcal{F}^{\mathrm s}$), we
repeatedly solve the linear program, apply the separation algorithm of
Section \ref{sec:sep} to its optimal solution, and add the returned inequality, until no
inequality is violated by more than $10^{-7}$ or the bound is within a relative
$10^{-6}$ of the exact value (computed with the extended formulation described below).
The linear program in Step 2 of the separation algorithm is solved by the simplex method,
which returns a basic optimal solution, and Steps 3 and 4 are carried out exactly as
stated. As a check, we also solved each instance
over Balas's extended formulation of $\D$ (with and without the side constraints): the
cutting-plane values agree with it to within a relative $6\cdot10^{-7}$, and every one
of the $12{,}964$ inequalities returned by the separation algorithm had a connected
equality graph. The experiments were implemented in Python using HiGHS~1.15.1 (simplex
method) and run on a single core of an Intel Xeon processor at 2.80\,GHz.

\medskip
\noindent {\bf Results.}
Figure \ref{fig:gapclosed} shows the mean fraction of gap closed over the five instances
in each cell. We make the following observations.
\begin{itemize}
\item The non-pairwise facets matter. Averaged over all instances, $\mathcal{F}$ closes
$68.6\%$ of the gap left by $\mathcal{L}$, whereas $\mathcal{D}$ closes $90.6\%$ (median
$95.9\%$). Equivalently, the facets that are not pairwise close on average $67\%$ (median
$84\%$) of the gap left by $\mathcal{F}$. The value of $\mathcal{D}$ is strictly better
than that of $\mathcal{F}$ on $92\%$ of the instances with $d\geq4$.
\item For $d=3$, the values of $\mathcal{F}$ and $\mathcal{D}$ coincide on every instance,
as predicted by Corollary \ref{cor:d3}.
\item The benefit of the complete description grows with both $d$ and $n$. The mean
difference between the gap closed by $\mathcal{D}$ and by $\mathcal{F}$ increases from
$0.11$ for $d=4$ to $0.32$ for $d=12$, and from $0.09$ for $n=2$ to about $0.25$ for
$n\geq10$. While $\mathcal{F}$ deteriorates as $d$ grows, $\mathcal{D}$ does not.
\item The picture is unchanged in the presence of side constraints: $\mathcal{L}^{\mathrm s}$
closes only $12\%$ of the gap, $\mathcal{F}^{\mathrm s}$ closes $71\%$, and
$\mathcal{D}^{\mathrm s}$ closes $92\%$ (median $97\%$).
\item Few facets are needed. Table \ref{tab:cuts} shows that the cutting-plane loop adds
at most $51$ inequalities (on average at most $14$ for each $d$) and takes well under a
second, although $\D$ can have up to $n+1+2^d\binom{n+d-1}{n}$ facets
(Theorem \ref{thm:facetcount}), about $3.5\cdot10^{11}$ for $d=12$ and $n=20$.
\end{itemize}
The gap that remains after $\mathcal{D}^{\mathrm s}$ is due to the trimming: closing it
would require the convex hull of the union of the trimmed polytopes $\Po^+_k$\,, which are
no longer generalized cross polytopes.

\begin{figure}[p]
\centering
\includegraphics[width=\textwidth]{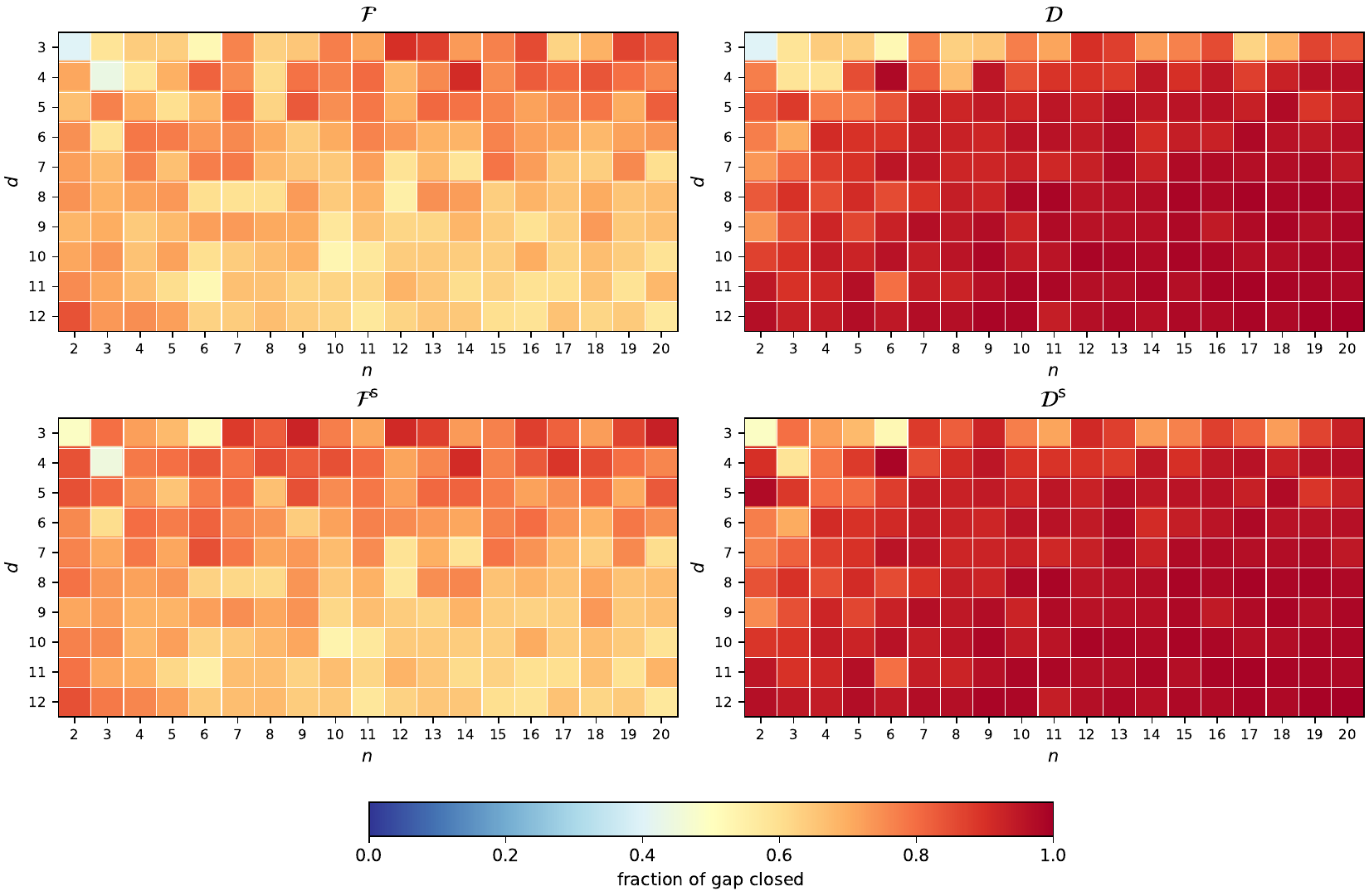}
\caption{Mean fraction of the gap left by $\mathcal{L}$ that is closed by $\mathcal{F}$,
$\mathcal{D}$, $\mathcal{F}^{\mathrm s}$ and $\mathcal{D}^{\mathrm s}$, over five
instances for each $(d,n)$.}
\label{fig:gapclosed}

\vspace{1.5\baselineskip}
\captionof{table}{Cutting-plane statistics by $d$, over all $n$ and instances; $|\mathcal{F}|$ is the
average number of pairwise inequalities, and the times are for the loop for $\mathcal{D}$.}
\label{tab:cuts}
\begin{tabular}{@{}rrrrrrrr@{}}
\toprule
 & & \multicolumn{2}{c}{cuts for $\mathcal{D}$} & \multicolumn{2}{c}{cuts for $\mathcal{D}^{\mathrm s}$} & \multicolumn{2}{c}{time (s)}\\
\cmidrule(lr){3-4}\cmidrule(lr){5-6}\cmidrule(l){7-8}
$d$ & $|\mathcal{F}|$ & avg & max & avg & max & avg & max\\
\midrule
3 & 93 & 0.0 & 0 & 0.0 & 0 & 0.004 & 0.012\\
4 & 174 & 1.6 & 11 & 1.6 & 13 & 0.011 & 0.093\\
5 & 255 & 3.3 & 10 & 3.1 & 11 & 0.019 & 0.062\\
6 & 336 & 4.4 & 14 & 4.2 & 15 & 0.029 & 0.164\\
7 & 417 & 6.4 & 20 & 6.1 & 21 & 0.041 & 0.164\\
8 & 498 & 8.0 & 28 & 8.1 & 34 & 0.055 & 0.235\\
9 & 579 & 8.5 & 27 & 8.4 & 30 & 0.065 & 0.268\\
10 & 660 & 10.8 & 34 & 10.6 & 34 & 0.081 & 0.324\\
11 & 741 & 12.1 & 42 & 12.0 & 37 & 0.108 & 0.436\\
12 & 822 & 13.7 & 51 & 13.6 & 48 & 0.126 & 0.543\\
\bottomrule
\end{tabular}
\end{figure}

\FloatBarrier

\medskip
\noindent {\bf Declaration of generative AI and AI-assisted technologies in the manuscript preparation process.}
During the preparation of this work, the authors used Claude to assist in developing and checking
the work. The authors take full responsibility for the content of the article.

\setlength{\bibsep}{0pt}
\bibliographystyle{elsarticle-num-names}
\bibliography{Bib}
\end{document}